\documentclass{amsart}

\usepackage{amssymb,amsmath,amsthm}
\usepackage{graphicx}

\newtheorem{thm}{Theorem}[section]
\newtheorem{lemma}[thm]{Lemma}
\newtheorem{prop}[thm]{Proposition}
\newtheorem{cor}[thm]{Corollary}

\newtheorem*{DKSUthm}{Theorem [DKSU]}

\def \R{\mathbb{R}}
\def \C{\mathbb{C}}
\def \Rn{\mathbb{R}^n}
\def \RnI{\mathbb{R}^{n+1}}
\def \RnIB{\mathbb{R}^{n+1} \setminus B}
\def \Rniip{\mathbb{R}^{n+1}_{1+}}

\def \a{\alpha}
\def \b{\beta}
\def \d{\delta}
\def \g{\gamma}
\def \th{\theta}

\def \s{\sigma}
\def \e{\varepsilon}
\def \ph{\varphi}

\def \L{\mathcal{L}}

\def \z{\zeta}

\def \nd{\partial_{\nu}}
\def \Om{\Omega}

\def \Lap{\triangle}
\def \grad{\nabla}

\def \S{\mathcal{S}}
\def \Czinf{C^{\infty}_0}
\def \Cinf{C^{\infty}}

\def \Omt{\tilde{\Omega}}
\def \frontom{F_{\Om}}
\def \backom{B_{\Om}}

\def \half{\frac{1}{2}}

\def \Laq{\mathcal{L}_{W,q}}
\def \Laqone{\mathcal{L}_{W_1,q_1}}
\def \Laqtwo{\mathcal{L}_{W_2,q_2}}

\def \Ltphe{\tilde{\mathcal{L}}_{\ph,\e}}
\def \Lph{\mathcal{L}_{\ph}}
\def \Lphe{\mathcal{L}_{\ph,\e}}
\def \Lphet{\tilde{\mathcal{L}}_{\ph,\e}}
\def \Lphets{\mathcal{L}_{\ph,\e,\s}}
\def \Lphaq{\mathcal{L}_{\ph,W,q}}
\def \Lphaqe{\mathcal{L}_{\ph,\e, W,q}}

\numberwithin{equation}{section}

\begin{document}

\title[A Partial Data Result for the Magnetic Schr\"{o}dinger Inverse Problem]{A Partial Data Result for the Magnetic Schr\"{o}dinger Inverse Problem}

\author[Chung]{Francis J. Chung}
\address{Department of Mathematics, University of Chicago, Chicago, IL 60637 USA}
\email{fjchung@math.uchicago.edu}



\begin{abstract}
This article shows that knowledge of the Dirichlet-Neumann (DN) map on certain subsets of the boundary for functions supported roughly on the rest of the boundary uniquely determines a magnetic Schr\"{o}dinger operator.  With some geometric conditions on the domain, either the subset on which the DN map is measured or the subset on which the input functions have support may be made arbitrarily small.  This is a response to a question posed in [DKSU].  The method involves modifying the Carleman estimate in that paper by conjugation with certain pseudodifferential-like operators.  
\end{abstract}

\maketitle

\section{Introduction}

Let $n \geq 2$, and let $\Om$ be a simply-connected bounded domain in $\RnI$, with smooth boundary.  If $W$ is a $C^2$ vector field on $\RnI$, and $q$ is an $L^{\infty}$ function on $\RnI$, then let $\Laq$ denote the magnetic Schr\"{o}dinger operator 
\[
\Laq = (D+W)^2 + q,
\]
where $D = -i\nabla$.  I will assume that $q$ and $W$ are such that zero is not an eigenvalue of $\Laq$ on $\Om$.  Then the Dirichlet problem 
\begin{eqnarray*}
& & \Laq u = 0 \\
& & u|_{\partial \Om} = g
\end{eqnarray*}
has a unique solution $u \in H^1(\Om)$ for each $g \in H^{\half}(\partial \Om)$.  Therefore for $g \in H^{\half}(\partial \Om)$ we can define the Dirichlet-Neumann map $\Lambda_{W,q}$ by
\[
\Lambda_{W,q} g = (\partial_{\nu} + iW \cdot \nu)u|_{\partial \Om},
\]
where $\nu$ is the outward unit normal, and $u$ is the unique solution to the Dirichlet problem with boundary value $g$.  This gives a well-defined map from $H^{\half}(\partial \Om)$ to $H^{-\half}(\partial \Om)$.  

The basic inverse problem associated to the magnetic Schr\"{o}dinger operator $\Laq$ is to recover $dW$ and $q$ from knowledge of $\Lambda_{W,q}$. (Here $W$ is identified with the 1-form $W_1 dx_{1} + \ldots + W_{n+1} dx_{n+1}$.)  Note that we cannot hope to recover $W$ itself since the Dirichlet-Neumann map is invariant under the gauge transformation $W \mapsto W + \grad \Psi$ whenever $\Psi \in C^1(\overline{\Om})$ and $\Psi|_{\partial \Om} = 0$.  However, identifying $dW$ identifies $W$ up to this gauge transformation.

In [Su], Sun first showed that full knowledge of the Dirichlet-Neumann map determines $dW$ and $q$ when $W$ is small enough, in a certain sense.  In [NSuU], Nakamura, Sun, and Uhlmann removed the smallness assumption, and showed that full knowledge of the Dirichlet-Neumann map determines $dW$ and $q$ for $C^2$ $W$ and $L^{\infty}$ $q$.  Tolmasky [To] and Salo [Sa1] improved the regularity conditions on $W$ to $C^{2/3 + \e}$ and Dini continuous, respectively.  Salo also gave in [Sa2] a proof for $W \in C^{1 + \e}$ involving a reconstruction method.  

In [DKSU], Dos Santos Ferreira, Kenig, Sj\"{o}strand, and Uhlmann proved a partial data result for this operator.  Assume that $x_0$ is not in the closure of the convex hull of $\Om$.  Define the front and back of $\partial \Om$ (with respect to $x_0$) by
\begin{eqnarray*}
F_{\Om} &=& \{x \in \partial \Om | (x - x_0) \cdot \nu(x) \leq 0\} \\
B_{\Om} &=& \{x \in \partial \Om | (x - x_0) \cdot \nu(x) \geq 0\}, 
\end{eqnarray*}
where $\nu(x)$ is the outward unit normal at $x$.  Then Theorem 1.1 in [DKSU] says (with different notation) the following.
\begin{DKSUthm}
Let $W_1$ and $W_2$ be $C^2$ vector fields on $\overline{\Om}$, and let $q_1$ and $q_2$ be $L^\infty$ functions on $\Om$.  Suppose $U$ is a neighbourhood of $F_{\Om}$ such that 
\[
\Lambda_{W_1,q_1} g |_{U} = \Lambda_{W_2,q_2} g |_{U}
\]
for all $g \in H^{\half}(\partial \Om)$.  Then $q_1 = q_2$ and $dW_1 = dW_2$.  
\end{DKSUthm}

However, in the case that $W \equiv 0$, [KSU] had already given a better partial data result, which says that we only need
\[
\Lambda_{0,q_1} g |_{U} = \Lambda_{0,q_2} g |_{U}
\] 
for $g \in H^{\half}(\partial \Om)$ with support in a neighbourhood of $B_{\Om}$, to conclude that $q_1 = q_2$.  

This paper will show that this sort of partial data result also holds for the magnetic Schr\"{o}dinger inverse problem.  

Define 
\[
Z_{\Om} = \{x \in \partial \Om | (x - x_0) \cdot \nu(x) = 0\}.
\]

The main results of this work are the following two theorems.

\begin{thm}\label{mainthm}
Let $W_1$ and $W_2$ be $C^2$ vector fields on $\overline{\Om}$, and let $q_1$ and $q_2$ be $L^\infty$ functions on $\Om$.  Let $U \subset \partial \Om$ be a neighbourhood of $F_{\Om}$, and let $E \subset \partial \Om$ be a compact subset of $F_{\Om} \setminus Z_{\Om}$.  Suppose
\[
\Lambda_{W_1,q_1} g |_{U} = \Lambda_{W_2,q_2} g |_{U}
\]
for all $g \in H^{\half}(\partial \Om)$ with support contained in $\overline{\partial \Om \setminus E}$.

Then $q_1 = q_2$, and $dW_1 = dW_2$.

\end{thm}

\begin{thm}\label{mainthm2}
Let $W_1$ and $W_2$ be $C^2$ vector fields on $\overline{\Om}$, and let $q_1$ and $q_2$ be $L^\infty$ functions on $\Om$.  Let $U \subset \partial \Om$ be a neighbourhood of $B_{\Om}$, and let $E \subset \partial \Om$ be a compact subset of $B_{\Om} \setminus Z_{\Om}$.  Suppose
\[
\Lambda_{W_1,q_1} g |_{U} = \Lambda_{W_2,q_2} g |_{U}
\]
for all $g \in H^{\half}(\partial \Om)$ with support contained in $\overline{\partial \Om \setminus E}$.

Then $q_1 = q_2$, and $dW_1 = dW_2$.

\end{thm}

The second theorem is essentially the first theorem after the conformal transformation on $\Om$ given by inversion in $x_0$.  

Roughly speaking, the first theorem says that if the Dirichlet-Neumann map is known on a neighbourhood of the front for functions supported on a neighbourhood of the back, then potentials can be determined.  The second says that if the Dirichlet-Neumann map is known on a neighbourhood of the back for functions supported on a neighbourhood of the front, then the potentials can be determined.  

If the domain $\Om$ is nice enough, then the front can be made arbitrarily small.  For example, if $\Om$ is strongly convex (convex, and the intersection of the boundary with any tangent hyperplane to the boundary consists only of one point), then the front can be contained in any non-empty open subset of the boundary.  This gives us the following corollary.

\begin{cor}
Suppose $\Om$ is a smooth bounded strongly convex domain in $\RnI$.  Let $W_1$ and $W_2$ be $C^2$ vector fields on $\overline{\Om}$, and let $q_1$ and $q_2$ be $L^\infty$ functions on $\Om$.  Then for any non-empty open subset $U$ of the boundary, there exists compact $E \subset U$ such that if 
\[
\Lambda_{W_1,q_1} g |_{U} = \Lambda_{W_2,q_2} g |_{U}
\]
for all $g \in H^{\half}(\partial \Om)$ with support contained in $\overline{\partial \Om \setminus E}$, then $q_1 = q_2$, and $dW_1 = dW_2$.

Alternatively, for any compact proper subset $E$ of the boundary, there exists $U \subset \partial \Om$ with $E \subset U$ such that if 
\[
\Lambda_{W_1,q_1} g |_{U} = \Lambda_{W_2,q_2} g |_{U}
\]
for all $g \in H^{\half}(\partial \Om)$ with support contained in $\overline{\partial \Om \setminus E}$, then $q_1 = q_2$, and $dW_1 = dW_2$.
\end{cor}
\vspace{4mm}

The first part of the corollary says that in particular, the Dirichlet-Neumann map can be measured on an arbitrarily small subset of the boundary.  The second part of the corollary says that alternatively, the input functions may be restricted to an arbitrarily small subset of the boundary.  

Theorem \ref{mainthm2} can either be proved from Theorem \ref{mainthm} by the change of variables mentioned above, or proved in the same manner as Theorem \ref{mainthm}, making the changes indicated at the end of section 6.  Therefore most of this paper will be devoted to proving Theorem \ref{mainthm}.  From here on, unless otherwise noted, I will assume $U$, $E$ and $\Om$ are as in Theorem \ref{mainthm}.  

The key to the proof of Theorem \ref{mainthm} is the construction of complex geometrical optics solutions to the system

\begin{equation}\label{pDirichletProblem}
\begin{split}
\Laq u  &= 0 \mbox{ on } \Om\\
u |_{E} &= 0.
\end{split} 
\end{equation} 

In [DKSU], these are constructed using a Carleman estimate for $\Laq$ and a Hahn-Banach argument.  However, the initial Carleman estimate proved in [DKSU] creates $L^2$ solutions.  These turn out to be good enough in the case $W=0$, but not when the magnetic potential is present.  Modifications to the Carleman estimate to create $H^1$ solutions in that paper destroy information about the behaviour of the solutions on the boundary, which explains the difference between the [KSU] and [DKSU] results.  Proving theorems \ref{mainthm} and \ref{mainthm2} will require more careful modification.  The Carleman estimate proved here for $\Laq$ can be described as follows. 

Let $\ph$ be a limiting Carleman weight on $\Om$; that is, a real-valued smooth function which has nonvanishing gradient on $\Om$ and satisfies
\[
\langle \ph'' \grad \ph, \grad \ph \rangle + \langle \ph'' \xi, \xi \rangle = 0
\]
whenever $|\xi| = |\grad \ph|$ and $\grad \ph \cdot \xi = 0$.  Define 
\[
\Lphaq = h^2 e^{\frac{\ph}{h}}\Laq e^{-\frac{\ph}{h}}
\]
Here $h$ is a semiclassical parameter; henceforth all Sobolev spaces and Fourier transforms in this note are semiclassical, unless otherwise specified, with $h$ being the semiclassical parameter.  For the rest of this paper, I will fix $\ph$ to be the logarithmic weight $\ph(x) = \log|x - x_0|$ unless otherwise stated.

I want to prove the following Carleman estimate. 
\begin{thm}\label{mainCarl}
There exists a smooth domain $\Om '$ with $\Om \subset \Om '$, and $E \subset \partial \Om '$, such that if $w \in \Czinf(\Om)$,
\[
h\| w\|_{L^2(\Om)} \lesssim \|\Lphaq w \|_{H^{-1}(\Om ')}.
\]
\end{thm}

Theorem \ref{mainCarl} will be proved over the next five sections.  In section 7, I will use this estimate to construct solutions to \eqref{pDirichletProblem}.  Once these are constructed, the proof of Theorem 1 follows by more or less the identical argument as in [DKSU].  That argument is presented in section 8 for completeness.  
\vspace{4mm}

\noindent \textbf{Acknowledgements}.  This research was partially supported by a Doctoral Postgraduate Scholarship from the Natural Science and Engineering Research Council of Canada.  The author would also like to thank Carlos Kenig for his guidance, support, and patience.

\section{An Initial Carleman Estimate}

%arxivfinal

I want to begin by considering a special version of Theorem \ref{mainCarl}, where the set $E$ coincides with a graph.  Without loss of generality, I will assume $x_0 = 0$.  We can equip $\RnI$ with spherical coordinates $(r,\th)$, with $r \in [0,\infty)$ and $\th \in S^n$.  

Define
\[
\Lphaqe = e^{\frac{\ph^2}{2\e}} \Lphaq e^{-\frac{\ph^2}{2\e}} 
\]

\begin{prop}\label{specCarl}
Suppose that $f : S^n \rightarrow (0,\infty)$ is a $\Cinf$ function such that $\Om$ lies entirely in the region $A_O = \{ (r,\th) | r \geq f(\th) \} \subset \RnI$, and $E$ is a subset of the graph $r = f(\th)$. (See the diagram below.)  If $w \in \Czinf(\Om)$, then
\[
\frac{h}{\sqrt{e}}\| w\|_{L^2(\Om)} \lesssim \|\Lphaqe w \|_{H^{-1}(A_O)}.
\]
\end{prop}

\includegraphics[scale=0.4]{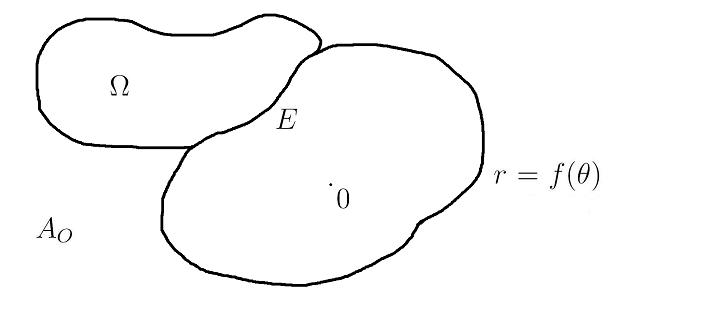}

In addition to $\Lphaq$, define
\begin{eqnarray*}
\Lphaqe &=& e^{\frac{\ph^2}{2\e}} \Lphaq e^{-\frac{\ph^2}{2\e}} \\
   \Lph &=& h^2 e^{\frac{\ph}{h}} \Lap e^{-\frac{\ph}{h}} \\
\mbox{ and } \Lphe &=& e^{\frac{\ph^2}{2\e}} \Lph e^{-\frac{\ph^2}{2\e}}. \\
\end{eqnarray*}
We will need to do some work with a set which is slightly larger than $\Om$, but still bounded.  Let $\Om_2 \subset A_O$ be a smooth bounded domain which contains $\Om$, such that $E \subset \partial \Om_2$, as indicated in the diagram.  Since $0$ lies outside the closure of the convex hull of $\Om$, I can pick $\Om_2$ so $0$ stays outside of the closure of the convex hull of $\Om_2$.  

\includegraphics[scale=0.4]{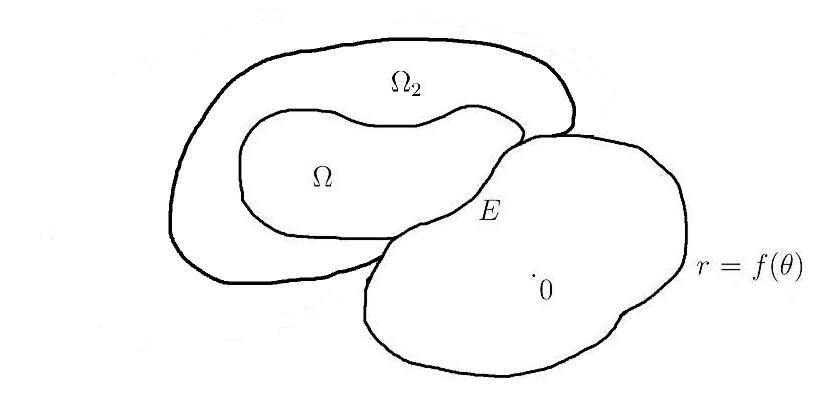}

\noindent We have the following Carleman estimate from [DKSU].  

\begin{lemma}\label{DKSUCarl}
If $w \in \Czinf(\Om_2)$, then 
\[
\frac{h}{\sqrt{\e}} \|w \|_{H^1(\Om_2)} \lesssim \|\Lphe w \|_{L^2(\Om_2)}.
\]
\end{lemma}
A note on inequalities here:  inequalities of the form $F(w,h) \lesssim G(w,h)$ mean that there exists $h_0 > 0$ independent of $w$, such that for $h \leq h_0$, the inequality $F(w,h) \leq CG(w,h)$ holds for some positive constant $C$ independent of $w$ and $h$.  In the case of Lemma \ref{DKSUCarl}, the constant implied in the $\lesssim$ sign is independent of $\e$ as well.

I can make a change of variables using the map $(r,\th) \mapsto (\frac{r}{f(\th)}, \th)$.  This is a diffeomorphism from $A_O$ to $\RnIB$, where $B$ is the open ball of radius $1$ centred at the origin, with inverse $(r,\th) \mapsto (rf(\th), \th)$.  Let $\Omt$ and $\Omt_2$ be the images of $\Om$ and $\Om_2$, respectively, under this map. This diffeomorphism maps $E$ to a part of the unit sphere $S^n$.  Note that since $0$ is outside of the closure of the convex hull of $\Om_2$, it is also outside of the closure of the convex hull of $\Omt_2$.  

\begin{lemma}\label{cvarCarl}
For $w \in \Czinf(\Omt_2)$, 
\begin{equation}\label{wigglyCarl}
\frac{h}{\sqrt{\e}} \|w \|_{H^1(\Omt_2)} \lesssim \|\Lphet w \|_{L^2(\Omt_2)}
\end{equation}
where
\begin{eqnarray*}
\Lphet &=& \left(1+ |\grad_{S^n}\log f(\th)|_{S^n}^2 \right) h^2 \partial^2_r \\
       & & - \frac{2}{r}\left( \a + (\grad_{S^n}\log f(\th)) \cdot_{S^n} h \grad_{S^n} \right) h \partial_r + \frac{1}{r^2}( \a^2 + h^2 \Lap_{S^n} ) \\
\end{eqnarray*}
and $\a = 1 + \frac{h}{\e}\log(r f(\th))$.  Here $\grad_{S^n}$ is the gradient operator on the unit sphere; $|\cdot|_{S^n}$ and $\cdot_{S^n}$ indicate the use of the Riemannian metric on $S^n$, and $\Lap_{S^n}$ is the Laplace-Beltrami operator on the unit sphere $S^n$.
\end{lemma}

\begin{proof}
Let $w \in \Czinf(\Om_2)$, and let 
\[
\tilde{w}(r,\th) = w(rf(\th), \th).  
\]
Then $\tilde{w} \in \Czinf(\Omt_2)$.  Now by a change of variables, 
\begin{eqnarray*}
\|\tilde{w}\|_{L^2(\Omt_2)} &=& \int_{S^n} \int_0^{\infty} |w(rf(\th),\theta)|^2 r^n dr \, d\th \\
                            &=& \int_{S^n} \int_{0}^{\infty} |w(r,\theta)|^2 r^n (f(\th))^{-n-1} dr \, d\th \\
                            &\simeq& \int_{S^n} \int_0^{\infty} |w(r,\theta)|^2 r^n dr \, d\th. \\
\end{eqnarray*}
Therefore
\begin{equation}\label{changevar}
\|\tilde{w}\|_{L^2(\Omt_2)} \simeq \|w\|_{L^2(\Om_2)}.
\end{equation}
The constants implied in the $\simeq$ sign depend on $f(\th)$.  In addition, 
\begin{eqnarray*}
\|\tilde{w}\|_{H^1(\Omt_2)}^2 &=& \|\tilde{w}\|_{L^2(\Omt_2)}^2 + \|h\grad \tilde{w}\|^2_{L^2(\Omt_2)} \\
                             &=& \|\tilde{w}\|_{L^2(\Omt_2)}^2 + \|h\partial_r \tilde{w}\|^2_{L^2(\Omt_2)}+\|h\grad_t \tilde{w}\|^2_{L^2(\Omt_2)},
\end{eqnarray*}
where $\grad_t$ is the orthogonal projection of $\grad$ onto the plane orthogonal to the radial direction.  Note that
\[
h\partial_r \tilde{w} = h f(\th) \widetilde{\partial_r w},
\]
and
\[
h\grad_t \tilde{w} = h\widetilde{\grad_t w} + h\widetilde{\partial_r w} \grad_t \frac{1}{f(\th)},
\]
so 

\begin{equation}\label{changevar2}
\|\tilde{w}\|_{H^1(\Omt_2)} \lesssim \|w\|_{H^1(\Om_2)}.
\end{equation}
Since $w(r,\th) = \tilde{w}(\frac{r}{f(\th)}, \th)$, the same argument shows that $\|\tilde{w}\|_{H^1(\Omt_2)} \gtrsim \|w\|_{H^1(\Om_2)}$, and therefore that 
\[
\|\tilde{w}\|_{H^1(\Omt_2)} \simeq  \|w\|_{H^1(\Om_2)}.
\]
where the constants implied in the $\simeq$ sign again depend on $f$.  

Now $\Lphe w \in L^2(\Om_2)$, so using the reasoning in \eqref{changevar} gives us that $\widetilde{\Lphe w} \in L^2(\Omt_2)$, and $\| \Lphe w\|_{L^2(\Om_2)} \simeq \|\widetilde{\Lphe w}\|_{L^2(\Omt_2)}$.  Therefore, by Lemma \ref{DKSUCarl}, 
\[
\frac{h}{\sqrt{\e}} \|\tilde{w} \|_{H^1(\Omt_2)} \lesssim \|\widetilde{\Lphe w} \|_{L^2(\Omt_2)}.
\]
Now a calculation shows that 
\begin{eqnarray*}
\Lphe &=& h^2 \partial_r^2 - r^{-1}\left( 2 - hn + 2\frac{h}{\e}\log r \right) h\partial_r \\
      & & + r^{-2}\left( 1+h-hn+\frac{h^2}{\e^2}((\log r)^2 - \e) + \frac{h^2}{\e}\log r + (2-hn)\frac{h}{\e}\log r +h^2\Lap_{S^n} \right).
\end{eqnarray*}
and then that
\[
\widetilde{\Lphe w} = f^{-2}(\th) \Ltphe \tilde{w} - hE \tilde{w}
\]
where $E$ is a first order semiclassical differential operator with coefficients which have bounds independent of $h$ and $\e$.  Therefore 
\[
\frac{h}{\sqrt{\e}} \|\tilde{w} \|_{H^1(\Omt_2)}  \lesssim \|f^{-2}(\th) \Ltphe \tilde{w} \|_{L^2(\Omt_2)} + h\|\tilde{w}\|_{H^1(\Omt_2)}.
\]

For small enough $\e$, the last term on the right side can be absorbed into the left side.  Also, $|f^{-2}|$ is bounded above, so 
\[
\frac{h}{\sqrt{\e}} \|\tilde{w} \|_{H^1(\Omt_2)}  \lesssim \|\Ltphe \tilde{w} \|_{L^2(\Omt_2)}.
\]
for all $w \in \Czinf(\Om_2)$.  Now any $w \in \Czinf(\Omt_2)$ can be written as $\tilde{v}$ for some $v \in \Czinf(\Om_2)$ just by taking $v(r,\th) = w(\frac{r}{f(\th)}, \th)$.  

This finishes the proof.  

\end{proof}

For the next step, we need to fix coordinates on $\RnIB$.  Since $\overline{\Omt_2}$ lies entirely on one side of a hyperplane through the origin, we can choose Cartesian coordinates $x_1, \ldots, x_{n+1}$  so that $\overline{\Omt_2}$ lies entirely in the intersection of $\RnIB$ with the halfspace $x_{n+1} > 0$.  We have a map $\s: \RnIB \cap \{ x_{n+1} > 0 \} \rightarrow [1, \infty) \times(0, \pi) \times \ldots \times (0, \pi) $ by 
\[
\s(x_1, \ldots, x_{n+1}) = (r, \th_1, \ldots, \th_n)
\]
where
\begin{eqnarray*}
x_1 &=& r\cos \th_1 \\
x_2 &=& r\sin \th_1 \cos \th_2 \\
\vdots & & \\
x_n &=& r\sin \th_1 \ldots \sin \th_{n-1} \cos \theta_n \\
x_{n+1} &=& r\sin \th_1 \ldots \sin \th_n.
\end{eqnarray*}
This fixes a set of spherical coordinates on $\RnIB \cap \{ x_{n+1} > 0 \}$.  Note that this map $\s$ is a diffeomorphism between the two open sets.  Since $\overline{\Omt_2}$ is compactly contained in $\RnIB \cap \{ x_{n+1} > 0 \}$ this diffeomorphism has bounded derivatives on $\Omt_2$.  

We will make another change of variables to the Carleman estimate in Lemma \ref{cvarCarl} to get a estimate for functions supported in $\s(\Omt_2)$.  For this we need some notation for expressing differential operators on $\RnIB \cap \{ x_{n+1} > 0 \}$ in spherical coordinates.  On the portion of the unit sphere $S^n$ in $\R^{n+1}$ which lies in the region $\{ x_{n+1} > 0 \}$, we can express the Riemannian metric in these coordinates.  Then the metric takes the form 
%a matrix!
\[ \left( \begin{array}{ccccc}
1      & 0              &        & \ldots & 0 \\
0      & (\sin \th_1)^2 & 0      & \ldots & 0 \\
\vdots & \vdots         & \vdots & \vdots & \vdots  \\
0      & \ldots         &        & 0      & (\sin \th_1 \ldots \sin \th_{n-1})^2 \end{array} \right)\] 
In these coordinates the metric has no dependence on $\th_n$.  Note that in the region near $\th_1 = \ldots = \th_{n-1} = \frac{\pi}{2}$, the metric is nearly the Euclidean metric.  In particular, if $\Lap_{S^n}$ is the Laplace-Beltrami operator on the sphere, then $h^2\Lap_{S^n}$, which has coordinate expression
\[
h^2 \partial_{\th_1}^2 + \frac{h^2}{\sin^2 \th_1}\partial_{\th_2}^2 + \ldots + \frac{h^2}{(\sin \th_1 \ldots \sin \th_{n-1})^2} \partial_{\th_n}^2 + hE
\]
on this domain, for some first order semiclassical differential operator $E$, will differ from $h^2 \Lap_{\th} = h^2\partial_{\th_1}^2 + \ldots +h^2\partial_{\th_n}^2$ by a second order semiclassical differential operator with small coefficients.

Functions in $g \in C^{\infty}_0(\Omt_2)$ can be pushed forward to functions in $C^{\infty}_0(\s(\Omt_2))$ by taking $g(\s^{-1}(r, \th_1, \ldots, \th_n))$.  It will be helpful to think of these pushed forward functions as functions on $\R^{n+1}_{1+} = \{ (r,\th) | \th \in \R^n, r \geq 1\}$. Now we can state the following corollary.

\begin{cor}
Let $\a = 1 + \frac{h}{\e}\log(r f(\th))$, $\b_f$ be a vector field on $\Rniip$ which agrees with the coordinate expression of $\grad_{S^n}\log f(\th)$ on $\s(\Omt_2)$, $\g_f$ be a function on $\Rniip$ which agrees with the coordinate expression of $|\grad_{S^n}\log f(\th)|_{S^n}$ on $\s(\Omt_2)$, and $L_{S^n}$ be a second order differential operator of $\Rniip$ which agrees with the coordinate expression of the Laplacian on the sphere on $\s(\Omt_2)$.

Let 
\begin{eqnarray*}
\Lphets &=& \left(1+ |\g_f|^2 \right) h^2 \partial^2_r \\
       & & - \frac{2}{r}\left( \a + \b_f \cdot h \grad_{\th} \right) h \partial_r + \frac{1}{r^2}( \a^2 + h^2 L_{S^n} ). \\
\end{eqnarray*}
Then for all $w \in \Czinf(\s(\Omt_2))$,  
\begin{equation}\label{flatCarl}
\frac{h}{\sqrt{\e}} \|w \|_{H^1(\s(\Omt_2))} \lesssim \|\Lphets w \|_{L^2(\s(\Omt_2))}
\end{equation}

\end{cor}

The proof follows from the same kind of argument made above.

The expression for $\Lphets$ is somewhat messy because of its dependence on $\th$.  Therefore I want to first work with a version where the functions of $\th$ that appear in $\Lphets$ are nearly constant in some sense.  Let $e_n$ denote the vector field on $S^n \cap \{ x_{n+1} > 0 \}$ given in coordinates by $(0, \ldots, 0, 1)$.   

\begin{prop}\label{simpleCarl}
Fix $K \geq 0$, and let $\mu > 0$.  Suppose that for all $\th$ such that some $(r,\th) \in \Om_2$, the following conditions hold:
\[
|\sin \th_j - 1| \leq \mu \mbox{ where } j = 1, \ldots, n-1
\]
and
\[
|\grad_{S^n}\log f - K e_n |_{S^n} \leq \mu.
\]
If $\mu$ is small enough, then for all $w \in \Czinf{\s(\Omt)}$,
\[
\frac{h}{\sqrt{\e}}\| w\|_{L^2(\s(\Omt))} \lesssim \|\Lphets w \|_{H^{-1}(\Rniip)}.
\]
\end{prop}

Note that the hypotheses imply that on $\s(\Omt_2)$, 
\[
|\b_f - (0, \ldots, 0, K)| \leq C_{\mu},
\]
\[
|\g_f - K| \leq C_{\mu},
\]
and if
\[
h^2 L_{S^n} = a_1 h^2\partial_{\th_1}^2 + \ldots + a_n h^2 \partial^2_{\th_n} + b_1 h^2\partial_{\th_1} + \ldots + b_n h^2\partial_{\th_n}
\]
then
\[
|a_j - 1 | \leq C_{\mu}
\]
for some constant $C_{\mu}$ which goes to zero if $\mu$ goes to zero.  $C_{\mu}$ may depend on $K$, but we are treating $K$ as fixed, so this will be ok.  

We may as well assume that $\b_f$, $\g_f$, and the coefficients of $L_{S^n}$ are extended to the rest of $\Rniip$ in such a way that these conditions continue to hold.

To prove this proposition, I want to divide $w$ into small and large frequency parts, and prove the estimate for each part separately.  Recall that $\Rniip = \{(r,\th)| \th \in \Rn, r \geq 1\}$.  Let $\S(\Rniip)$ be the restrictions to $\Rniip$ of Schwartz functions on $\RnI$.  Note that functions in $\Czinf(\s(\Omt_2))$ are in $\S(\Rniip)$.   

Let $r_1$ and $r_2$ be such that 
\[
\frac{|K|^2}{1 + |K|^2} < r_1 < r_2 \leq \half + \frac{|K|^2}{2(1 + |K|^2)} < 1,
\] 
and let $\d_1$ and $\d_2$ be such that $\d_2 > \d_1 > 0$.  Let $\rho \in \Czinf(\Rn)$ be a cutoff function such that $\rho(\xi) = 0$ if $|\xi|^2 > r_2$ or $|\xi_n| > \d_2$, and $\rho(\xi) = 1$ if $|\xi|^2 \leq r_1$ or $|\xi_n| \leq \d_1$.  

Let the hat $\hat{ }$ indicate the (semiclassical) Fourier transform in the $\th$ variables only.  (In general, Fourier transforms here will be in the $\th$ variables only unless otherwise indicated.)  For $w \in \Czinf(\s(\Omt))$, define $w_s$ and $w_{\ell}$ by $\hat{w}_s = \rho(\xi) \hat{w}$ and $\hat{w}_{\ell} = (1 - \rho(\xi))\hat{w}$, so $w = w_s + w_{\ell}$.  

\begin{lemma}\label{smallCarl}
There exists $\mu > 0$ and choices of $r_1, r_2,\d_1,$ and $\d_2$ such that if the hypotheses of Proposition \ref{simpleCarl} hold, then 
\[
\frac{h}{\sqrt{\e}}\|w_s\|_{L^2(\Rniip)} \lesssim \|\Lphets w_s \|_{H^{-1}(\Rniip)} + h \|w\|_{L^2(\s(\Omt))},
\]
for all $w \in \Czinf(\s(\Omt))$, where $w_s$ is defined as above.
\end{lemma}

\begin{lemma}\label{largeCarl}
There exists $\mu > 0$ such that if the hypotheses of Proposition \ref{simpleCarl} hold, then 
\[
\frac{h}{\sqrt{\e}}\| w_{\ell}\|_{L^2(\Rniip)} \lesssim \|\Lphets w_{\ell} \|_{H^{-1}(\Rniip)}+ h \|w\|_{L^2(\s(\Omt))},
\]
for all $w \in \Czinf(\s(\Omt))$, where $w_{\ell}$ is defined as above.
\end{lemma}

Taken together, these two lemmas imply Proposition \ref{simpleCarl}.  To see why, first we need a lemma.

\begin{lemma}
If $A$ is a pseudodifferential operator of order $m \geq 0$ on $\R^n$, it can be applied to Schwartz functions on $\R^{n+1}$, by taking $Af(x_1, \ldots, x_{n+1})$ for each fixed $x_{n+1}$.  Then as an operator defined for functions on $\R^{n+1}$ it extends to an operator from $H^{k+m}(\RnI)$ to $H^k(\RnI)$, 
\[
\| Af \|_{H^k(\R^{n+1})} \lesssim \| f \|_{H^{k+m}(\R^{n+1})}
\]
\end{lemma}

\begin{proof}
I will give a proof for the case where $k$ is a non-negative integer.  Let $(x,y); x \in \Rn, y \in \R$ denote the coordinates on $\RnI$.  Then if $f$ is Schwartz,
\[
\| Af \|^2_{H^k(\R^{n+1})} = \sum_{0 \leq |\a| + j \leq k} \|h^{|\a| + j}\partial^{\a}_x \partial^j_y Af \|^2_{L^2(\R^{n+1})}
\]
Now
\[
\|h^{|\a| + j} \partial^{\a}_x \partial^j_y Af \|^2_{L^2(\R^{n+1})} = \int_{-\infty}^{\infty} \int_{\Rn} |h^{|\a| + j} \partial^{\a}_x \partial^j_y Af |^2 dx \, dy.
\]
Here $A$ and $\partial^j_y$ commute, so 
\begin{eqnarray*}
\|h^{|\a| + j} \partial^{\a}_x \partial^j_y Af \|^2_{L^2(\R^{n+1})} &=& \int_{-\infty}^{\infty} \int_{\Rn} |h^{|\a|}\partial^{\a}_x A (h^{j}\partial^j_y f) |^2 dx \, dy\\
                                                       &\lesssim& \int_{-\infty}^{\infty} \|A (h^{j}\partial^j_y f) \|_{H^{|\a|}(\Rn)}^2  dy. \\
\end{eqnarray*}
Therefore, by the boundedness of $A$, 
\[
\|h^{|\a| + j} \partial^{\a}_x \partial^j_y Af \|^2_{L^2(\R^{n+1})} \lesssim \int_{-\infty}^{\infty} \|(h^{j}\partial^j_y f) \|_{H^{|\a|+m}(\Rn)}^2  dy.
\]
Since $|\a| + m \geq 0$, 
\[
\|h^{|\a| + j} \partial^{\a}_x \partial^j_y Af \|^2_{L^2(\R^{n+1})} \lesssim \| (h^{j}\partial^j_y f) \|_{H^{|\a| + m}(\RnI)}^2.
\]
Therefore
\begin{eqnarray*}
\|h^{|\a| + j} \partial^{\a}_x \partial^j_y Af \|^2_{L^2(\R^{n+1})} &\lesssim& \| f \|_{H^{|\a|+ j + m}(\RnI)}^2 \\
                                                       &\lesssim& \| f \|_{H^{k + m}(\RnI)}^2, \\
\end{eqnarray*}
so
\[
\| Af \|^2_{H^k(\R^{n+1})} \lesssim \| f \|_{H^{k + m}(\RnI)}^2.
\]

\end{proof}

Note that this still holds if $A$ depends on $y$, as long as $A$ has bounds uniform in $y$.  

%************************************************************************************************************************************************************************

\begin{proof}[Proof of Proposition \ref{simpleCarl}]
Adding together the estimates from Lemma \ref{smallCarl} and Lemma \ref{largeCarl} gives
\begin{eqnarray*}
& &        \frac{h}{\sqrt{\e}}(\| w_s\|_{L^2(\Rniip)}+\| w_{\ell}\|_{L^2(\Rniip)}) \\
&\lesssim& \|\Lphets w_s \|_{H^{-1}(\Rniip)}+\|\Lphets w_{\ell} \|_{H^{-1}(\Rniip)} +h\|w\|_{L^2(\s(\Omt))} \\
\end{eqnarray*}
Since $ w_s + w_{\ell} = w$,
\[
\frac{h}{\sqrt{\e}}\| w\|_{L^2(\s(\Omt))} \lesssim \|\Lphets w_s \|_{H^{-1}(\Rniip)} + \|\Lphets w_{\ell} \|_{H^{-1}(\Rniip)}+ h \|w\|_{L^2(\s(\Omt))}.
\]
For small enough $\e$, we can absorb the last term into the left side to give

\[
\frac{h}{\sqrt{\e}}\| w\|_{L^2(\s(\Omt))} \lesssim \|\Lphets w_s \|_{H^{-1}(\Rniip)} + \|\Lphets w_{\ell} \|_{H^{-1}(\Rniip)}.
\]

Since $(1+|\g_f|^2) > 1 + K^2 - C_{\mu}$, for $\mu$ small enough, we have 
\[
\frac{h}{\sqrt{\e}}\| w\|_{L^2(\s(\Omt))} \lesssim \|(1+|\g_f|^2)^{-1}\Lphets w_s \|_{H^{-1}(\Rniip)} + \|(1+|\g_f|^2)^{-1} \Lphets w_{\ell} \|_{H^{-1}(\Rniip)}.
\]

Now $w_s = Pw$, where $P$ is the semiclassical pseudodifferential operator of order 0 on $\Rn$ with symbol $\rho(\xi)$.  $P$ commutes with $\partial_r$.  Therefore there are some operators $E_1$ and $E_0$ which for each fixed $r \in [1, \infty)$ are semiclassical pseudodifferential operators of order $1$ and $0$ respectively, on $\Rn$ such that 
\begin{eqnarray*}
& &        \|(1+|\g_f|^2)^{-1} \Lphets w_s \|_{H^{-1}(\Rniip)} = \|(1+|\g_f|^2)^{-1} \Lphets P w \|_{H^{-1}(\Rniip)} \\
&\lesssim& \| P (1+|\g_f|^2)^{-1}  \Lphets w \|_{H^{-1}(\Rniip)} + h\|E_0 h\partial_r + E_1 w\|_{H^{-1}(\Rniip)} \\
\end{eqnarray*}
There is no $hE_{-1}h^2\partial_r^2$ in the error term because the coefficient of $h^2\partial_r^2$ in $(1+|\g_f|^2)^{-1} \Lphets$ is just $1$.  

By the lemma above, $E_1^{*}$ is bounded from $H^1_0(\Rniip)$ to $L^2(\Rniip)$, so by duality, $E_1$ is bounded from $L^2(\Rniip)$ to $H^{-1}(\Rniip)$.  

In addition, $E_0^{*}$ is bounded from $H^1(\Rniip)$ to $H^1(\Rniip)$.  Moreover, $E_0^{*}$ takes functions with trace $0$ on the boundary of $\Rniip$ to other functions with trace $0$ on the boundary of $\Rniip$, so by duality, $E_0$ is bounded from $H^{-1}(\Rniip)$ to $H^{-1}(\Rniip)$.
These bounds must be uniform in $r$ for the range of $r$ allowed on the support of $w$. Therefore
\[
\|(1+|\g_f|^2)^{-1} \Lphets w_s \|_{H^{-1}(\Rniip)} \lesssim \|P (1+|\g_f|^2)^{-1}  \Lphets w \|_{H^{-1}(\Rniip)} + h\|w\|_{L^2(\Rniip)} 
\]
Now by the same lemma, $P$ is bounded from $H^1(\Rniip)$ to $H^1(\Rniip)$.  Moreover, if $u$ has trace zero on the boundary of $\Rniip$, then so does $Pu$, so $P$ is bounded from $H^1_0(\Rniip)$ to $H^1_0(\Rniip)$.  Since $\rho$ is real valued, $P$ is also self adjoint, so by duality $P$ is bounded from $H^{-1}(\Rniip)$ to $H^{-1}(\Rniip)$.  Therefore
\[
\|(1+|\g_f|^2)^{-1} \Lphets w_s \|_{H^{-1}(\Rniip)} \lesssim \|(1+|\g_f|^2)^{-1}  \Lphets w \|_{H^{-1}(\Rniip)} + h\|w\|_{L^2(\Rniip)}.
\]
and thus
\[
\|\Lphets w_s \|_{H^{-1}(\Rniip)} \lesssim \| \Lphets w \|_{H^{-1}(\Rniip)} + h\|w\|_{L^2(\Rniip)}.
\]
Similarly,
\[
\|\Lphets w_{\ell} \|_{H^{-1}(\Rniip)} \lesssim \|\Lphets w \|_{H^{-1}(\Rniip)} + h \|w\|_{L^2(\Rniip)}.
\]
Therefore
\[
\frac{h}{\sqrt{\e}}\| w\|_{L^2(\s(\Omt))} \lesssim \|\Lphets w \|_{H^{-1}(\Rniip)} + h\|w\|_{L^2(\Rniip)}.
\]
Again the last term can be absorbed into the left side for small enough $\e$, so

\[
\frac{h}{\sqrt{\e}}\| w\|_{L^2(\s(\Omt))} \lesssim \|\Lphets w \|_{H^{-1}(\Rniip)}
\]
for each $w \in \Czinf(\s(\Omt))$ as desired.  

\end{proof}

Therefore we need only deal with the proofs of Lemmas \ref{smallCarl} and \ref{largeCarl}.

\section{Small Frequency Operators}

I want to begin by describing some operators for use in proving the small frequency case.

Consider the function $F:\R^n  \rightarrow \mathbb{C}$ given by
\[
\overline{F(\xi)} = \frac{1}{1 + |K|^2}\left( 1 + iK\xi_n + \sqrt{2iK\xi_n-(K\xi_n)^2 +(1 + |K|^2)|\xi|^2 -|K|^2} \right).
\]
where the square root is taken to mean the branch of the square root function with nonnegative imaginary part.  If $r_2$ and $\d_2$ are chosen small enough, then this is nearly continuous on the support of $\rho$.  To be more precise, $F$ is smooth except where 
\[
\tau_K(\xi) = 2iK\xi_n-(K\xi_n)^2 +(1 + |K|^2)|\xi|^2 -|K|^2
\]
lies on the nonnegative real axis, where this branch of the square root has its branch cut.  This occurs when $\xi_n = 0$ and $|\xi|^2 \geq \frac{|K|^2}{1 + |K|^2}$, and gives a jump discontinuity of size $2\sqrt{(1 + |K|^2)|\xi|^2 -|K|^2}.$  However, on the support of $\rho$, $|\xi|^2 \leq r_2$, so for $r_2$ close to $\frac{|K|^2}{1 + |K|^2}$ the maximum possible size of the jump discontinuity is small.  

Therefore for any $\d>0$, we can pick a smooth function $F_s(\xi)$ such that 
\[
\left| F_s(\xi) - F(\xi) \right| \leq \d
\]
on the support of $\rho$, by choosing $r_2$ small enough.  Note that the derivatives of $F_s$ may depend on $r_1$, $r_2$, $\d_1$, and $\d_2$.  Since the choice of these in turn depends on $\d$, the derivatives of $F_s$ are bounded by a quantity that depends on $\d$.  

Now consider the necessary bounds on $F_s$.  On the support of $\rho$, the imaginary part of $\tau_K$ must lie in the interval $[-2K\d_2, 2K\d_2]$.  The real part of $\tau_{K}$ is given by 
\[
- (K\xi_n)^2 - |K|^2 + (1 + |K|^2)|\xi|^2.
\]
We have that $|\xi|^2 \leq r_2$ on the support of $\rho$.  We can choose $r_2$ so close to $\frac{K^2}{1 + K^2}$ that 
\[
(1 + K^2)r_2 - K^2 \leq \d_2.
\]
Then the real part of $\tau_K$ is bounded above by $\d_2$ on the support of $\rho$.  Therefore on the support of $\rho$, $(\mathrm{Re} (\tau_K),\mathrm{Im} (\tau_K)) \in (-\infty, \d_2]\times [-2K\d_2, 2K\d_2]$, and so by taking $\d_2$ small enough, we can ensure that the real part of $\sqrt{\tau_K}$ has absolute value less than $\frac{1}{3}$ on the support of $\rho$.

Therefore if $\d$ is small enough, $\mathrm{Re}(F_s), |F_s| > \frac{1}{2+2K^2}$ on the support of $\rho$.  We can define $F_s$ outside the support of $\rho$ so that $\mathrm{Re}(F_s), |F_s| > \half \frac{1}{1+|K|^2}$ for all $\xi$, and $\mathrm{Re}(F_s), |F_s| \simeq 1 + |\xi|$ for large $|\xi|$.  

Now for $u \in \S(\Rniip)$, define $J_{s}u$ by 
\[
\widehat{J_{s} u}(r, \xi) = \left( \frac{F_s(\xi)}{r} + h\partial_r \right) \hat{u}(r, \xi).
\]
This operator has adjoint $J_{s}^{*}$ given by
\[
\widehat{J_{s}^{*} u}(r, \xi) = \left( \frac{\overline{F_s(\xi)}}{r} - h\partial_r \right) \hat{u}(r, \xi).
\]
These operators have (right) inverses defined by 

\[
\widehat{J_{s}^{-1}u} (r,\xi) = h^{-1} \int_1^{r} \hat{u}(t,\xi) \left( \frac{t}{r} \right)^{\frac{F_s(\xi)}{h}} dt 
\]
and
\[
\widehat{J_{s}^{*-1}u} (r,\xi) = h^{-1} \int_r^{\infty} \hat{u}(t,\xi) \left( \frac{r}{t} \right)^{\frac{\overline{F_s(\xi)}}{h}} dt.
\]
Each of these is well defined for functions in $\S(\Rniip)$.

Define the weighted Sobolev space $H^1_r(\Rniip)$ by the norm 
\[
\| u \|^2_{H^1_r(\Rniip)} = \left\| \frac{u}{r} \right\|^2_{L^2(\Rniip)} + \|h\partial_r u\|^2_{L^2(\Rniip)} + \left\| \frac{h}{r}\grad_{\th} u \right\|^2_{L^2(\Rniip)}.
\]

Note that since $\s(\Omt_2)$ lies in the set $1 \leq r \leq C_{\s(\Omt_2)}$ for some $C_{\s(\Omt_2)}$, $H^1$ and $H^1_r$ norms are comparable for functions supported on $\s(\Omt_2)$, with constants of comparability depending only on $C_{\s(\Omt_2)}$.  This holds more generally for any functions supported in $1 \leq r \leq C_{\s(\Omt_2)}$.

\begin{lemma}\label{smallJbounds}
$J_s, J_s^{*},J_s^{-1},$ and $J_s^{*-1}$ extend as bounded maps
\[
J_{s}, J_s^{*}: H^1_r(\Rniip) \rightarrow L^2(\Rniip)
\]
and
\[
J_{s}^{-1}, J_s^{*-1}: L^2(\Rniip) \rightarrow H^1_r(\Rniip).
\]
Moreover, the extensions of $J_{s}^{*}$ and $J^{*-1}_{s}$ are isomorphisms.  
\end{lemma}

\begin{proof}

Consider $J_{s}$ first.  If $u \in \S(\Rniip)$, then 
\begin{eqnarray*}
\|J_{s}u \|^2_{L^2(\Rniip)} &=& h^{-n} \| \widehat{J_{s}u} \|^2_{L^2(\Rniip)} \\
                            &=& h^{-n} \left\| \frac{F_s(\xi)}{r} \hat{u} + h\partial_r \hat{u} \right\|^2_{L^2(\Rniip)} \\
                            &\leq& h^{-n} \left\| \frac{F_s(\xi)}{r} \hat{u} \right\|^2_{L^2(\Rniip)} + h^{-n}\|h\partial_r \hat{u} \|^2_{L^2(\Rniip)} \\
                            &\lesssim& h^{-n} \left\| \frac{1 + |\xi|}{r} \hat{u} \right\|^2_{L^2(\Rniip)} + \|h\partial_r u \|^2_{L^2(\Rniip)} \\
                            &\lesssim& \left\| \frac{u}{r} \right\|^2_{L^2(\Rniip)} + \left\| \frac{h}{r}\grad_{\th}u \right\|^2_{L^2(\Rniip)} + \|h\partial_r u \|^2_{L^2(\Rniip)} \\
                            &=& \|u\|_{H^1_r(\Rniip)} \\
\end{eqnarray*}

By the density of $\S(\Rniip)$ in $H^1_r(\Rniip)$, $J_{s}$ extends to a bounded map $J_{s}: H^1_r(\Rniip) \rightarrow L^2(\Rniip)$.  
The proof for $J_s^{*}$ is similar.

Now consider $J^{-1}_{s}$.  If $u \in \S(\Rniip)$, then
\begin{eqnarray*}
\int_1^{\infty} \left| \frac{1}{r}\widehat{J^{-1}_{s}u} \right|^2 dr &=& \int_1^{\infty} \left| h^{-1} \int_1^r \hat{u}(t,\xi) \left( \frac{t}{r} \right)^{\frac{F_s(\xi)}{h}}dt \right|^2 r^{-2} dr \\
&\leq& \int_1^{\infty} \left| h^{-1} \int_0^r \hat{u}(t,\xi) \left( \frac{t}{r} \right)^{\frac{F_s(\xi)}{h}}dt \right|^2 r^{-2} dr. \\
\end{eqnarray*}
By a change of variables, we get
\[
\int_1^{\infty} \left| \frac{1}{r}\widehat{J^{-1}_{s}u} \right|^2 dr = \int_1^{\infty} \left| h^{-1} \int_0^1 \hat{u}(rt, \xi) t^{\frac{F_s(\xi)}{h}}dt \right|^2 dr.
\]
Then Minkowski's inequality gives us 
\begin{eqnarray*}
\int_1^{\infty} \left| \frac{1}{r}\widehat{J^{-1}_{s}u} \right|^2 dr &\leq& h^{-2} \left( \int_0^1 \left( \int_1^{\infty} |\hat{u}(rt, \xi) t^{\frac{F_s(\xi)}{h}}|^2 dr \right)^{\half} dt \right)^2 \\
&=& h^{-2} \left( \int_0^1 \left( \int_1^{\infty} |\hat{u}(rt,\xi)|^2 dr \right)^{\half} t^{\frac{\mathrm{Re}F_s(\xi)}{h}} dt \right)^2. \\
\end{eqnarray*}
Changing variables again, we get
\begin{eqnarray*}
& & \int_1^{\infty} \left| \frac{1}{r}\widehat{J^{-1}_{s}u} \right|^2 dr \\
&\leq& h^{-2} \left( \int_0^1 \left( \int_1^{\infty} |\hat{u}(r, \xi)|^2  dr \right)^{\half} t^{\mathrm{Re}\frac{F_s(\xi)}{h}} t^{-\half}dt \right)^2 \\
&=& h^{-2} \int_1^{\infty} |\hat{u}(r, \xi)|^2 dr \left( \frac{h}{\mathrm{Re}F_s(\xi)+\frac{h}{2}} \right)^2\\
&\simeq& \int_1^{\infty} \left| \frac{\hat{u}(r,\xi)}{1 + |\xi|} \right|^2 dr \\
&\leq& \int_1^{\infty} |\hat{u}(r, \xi)|^2 dr\\
\end{eqnarray*}
Therefore
\begin{eqnarray*}
\left\| \frac{1}{r}J^{-1}_{s}u \right\|^2_{L^2(\Rniip)} &=& h^{-n} \left\| \frac{1}{r}\widehat{J^{-1}_{s}u} \right\|^2_{L^2(\Rniip)} \\
                                                        &=& h^{-n} \int_{\R^n} \int_1^{\infty} \left| \frac{1}{r}\widehat{J^{-1}_{s}u} \right|^2 dr \, d \xi \\
                                                        &\lesssim& h^{-n} \int_{\R^n} \int_1^{\infty} |\hat{u}|^2 dr \, d \xi \\
                                                        &=& h^{-n} \|\hat{u}\|^2_{L^2(\Rniip)} \\
                                                        &=& \|u\|^2_{L^2(\Rniip)} \\
\end{eqnarray*}
Similarly,
\[
\int_1^{\infty} \left| \frac{\xi}{r}\widehat{J^{-1}_{s}u} \right|^2 dr \lesssim  \int_1^{\infty} |\hat{u}(r, \xi)|^2 dr
\]
so
\[
\left\| \frac{h}{r} \grad_{\th} J^{-1}_{s}u \right\|^2_{L^2(\Rniip)} \lesssim \|u\|^2_{L^2(\Rniip)}.
\]
Finally, 
\[
h \partial_r \widehat{J^{-1}_{s}u} = -\left( \frac{F_s(\xi)}{r} \right)  \widehat{J^{-1}_{s}u} + \hat{u}
\]
so
\[
\int_1^{\infty} \left| h\partial_r \widehat{J^{-1}_{s}u} \right|^2 dr \lesssim  \int_1^{\infty} |\hat{u}(r, \xi)|^2 dr
\]
and
\[
\left\| h\partial_r J^{-1}_{s}u \right\|^2_{L^2(\Rniip)} \lesssim \|u\|^2_{L^2(\Rniip)}
\]
by the same logic.  

Putting all of this together gives
\[
\| J^{-1}_{s} u \|^2_{H^1_r(\Rniip)} \lesssim \|u\|^2_{L^2(\Rniip)}
\]
for $u \in \S(\Rniip)$.  Therefore by the density of $\S(\Rniip)$ in $L^2(\Rniip)$, $J^{-1}_{s}$ extends to a bounded map 
\[
J^{-1}_{s}: L^2(\Rniip) \rightarrow H^1_r(\Rniip).
\]
Again the proof for $J^{*-1}_{s}$ is similar.

It remains to show that the extensions of $J_{s}^{*}$ and $J^{*-1}_{s}$ are isomorphisms.  Note that if $u \in \S(\Rniip)$, then
\begin{eqnarray*}
\mathcal{F}(J_{s}^{*}J^{*-1}_{s}u) &=& \left(\frac{\overline{F_s(\xi)}}{r} - h\partial_r \right) \widehat{J^{*-1}_{s}u}(r,\xi) \\
              &=& \left(\frac{\overline{F_s(\xi)}}{r}-h\partial_r\right) h^{-1}\int_r^{\infty}\hat{u}(t,\xi) \left( \frac{r}{t}\right)^{\frac{\overline{F_s(\xi)}}{h}} dt \\
              &=& \left(\frac{\overline{F_s(\xi)}}{r} \right) \widehat{J^{*-1}_{s}u} - \left(\frac{\overline{F_s(\xi)}}{r} \right) \widehat{J^{*-1}_{s}u} +\hat{u}(r,\xi) \\
              &=& \hat{u}(r,\xi). \\
\end{eqnarray*}
(Here $\mathcal{F}$ stands for the semiclassical Fourier transform in the $\theta$ variables, just like the hat $\hat{ }$.  I will use this notation when the hat becomes unwieldy.)  Therefore
\begin{equation}\label{JJminus}
J_{s}^{*}J^{*-1}_{s}u = u
\end{equation}
for all $u \in \S(\Rniip)$.  

On the other hand, integration by parts gives
\begin{eqnarray*}
\mathcal{F}(J^{*-1}_{s}J_{s}^{*}u) &=& h^{-1}\int_r^{\infty} \left( \frac{\overline{F_s(\xi)}}{t}-h\partial_t \right)\hat{u}(t,\xi)\left( \frac{r}{t}\right)^{\frac{\overline{F_s(\xi)}}{h}}dt \\
                      &=& h^{-1}\int_r^{\infty} \left( \frac{\overline{F_s(\xi)}}{t}\right)\hat{u}(t,\xi) \left( \frac{r}{t}\right)^{\frac{\overline{F_s(\xi)}}{h}} dt \\
                      & & + h^{-1}\int_r^{\infty} \hat{u}(t) \left( \frac{r}{t}\right)^{\frac{\overline{F_s(\xi)}}{h}}\left( -\frac{\overline{F_s(\xi)}}{t} \right) dt \\
                      & & \left. -\hat{u}(t,\xi) \left( \frac{r}{t}\right)^{\frac{\overline{F_s(\xi)}}{h}} \right|_{t = r}^{\infty}\\
                      &=& \hat{u}(r,\xi),
\end{eqnarray*}
so
\begin{equation}\label{JminusJ}
J^{*-1}_{s}J_{s}^{*}u = u
\end{equation}
for all $u \in \S(\Rniip)$.  

Now consider $J^{*-1}_{s}$.  By \eqref{JJminus},
\[
\|J^{*-1}_{s} u\|_{H^1_r(\Rniip)} \gtrsim \|u\|_{L^2(\Rniip)}
\]
for all $u \in \S(\Rniip)$.  By the density of $\S(\Rniip)$ in $L^2(\Rniip)$, this holds for all $u \in L^2(\Rniip)$.  Therefore $J^{*-1}_{s}$ is 1-1 with closed range.  Now \eqref{JminusJ} implies that $\S(\Rniip)$ is in the range of $J^{*-1}_{s}$, so by the density of $\S(\Rniip)$ in $H^1_r(\Rniip)$, $H^1_r(\Rniip)$ is in the range of $J^{*-1}_{s}$.  Therefore $J^{*-1}_{s}: L^2(\Rniip) \rightarrow H^1_r(\Rniip)$ is an isomorphism.  

Similarly, \eqref{JminusJ} shows that 
\[
\|J_{s}^{*} u\|_{L^2(\Rniip)} \gtrsim \|u\|_{H^1_r(\Rniip)}
\]
for all $u \in \S(\Rniip)$, and hence for all $u \in H^1_r(\Rniip)$.  Therefore $J_{s}^{*}$ is 1-1 with closed range.  Then \eqref{JJminus} implies that $\S(\Rniip)$ is in the range of $J_{s}^{*}$, and so $L^2(\Rniip)$ is in the range of $J_{s}^{*}$.  Therefore  $J_{s}^{*}: H^1_r(\Rniip) \rightarrow L^2(\Rniip)$ is also an isomorphism. 

Note that $J_{s}^{-1} J_{s} u \neq u$ in general, because integration by parts will pick up a boundary term at $r = 1$.  Therefore the extensions of $J_{s}$ and $J_{s}^{-1}$ are not isomorphisms.

\end{proof}

Let $H^1_{r,0}(\Rniip)$ denote the subspace of $H^1_r(\Rniip)$ consisting of functions with trace zero on the hyperplane $r=1$, and let $H^{-1}_r(\Rniip)$ denote the dual space to $H^1_{r,0}(\Rniip)$.  

The operator $J_{s}$ is closely related to a pseudodifferential operator.  In particular, it has the following properties, which will be needed later.  

\begin{lemma}\label{smallJprops}

i) Suppose $w \in \S(\Rniip) $.  Then if $Q$ is a second order semiclassical differential operator with smooth bounded coefficients on $\Rniip$, then
\[
\|(J_{s} Q - Q J_{s})w\|_{H^{-1}_r(\Rniip)} \lesssim hC_{\d}\|rw\|_{H^1(\Rniip)}
\] 

ii) Suppose $w \in \S(\Rniip) $.  Let $\chi \in \S(\Rniip)$.  Then
\[
\| J_{s} \chi J^{-1}_{s}w \|_{L^2(\Rniip)} \gtrsim \|\chi w\|_{L^2(\Rniip)} - hC_{\d}\| r w\|_{L^2(\Rniip)}
\]

\end{lemma}

The $C_{\d}$ factor is written explicitly to track the $\d$ dependence.  

\begin{proof}

i)  
First, note that multiplication by $1/r$ is a bounded operator from $H^1_{r,0}(\Rniip)$ to $H^1_0(\Rniip)$.  Therefore by duality, it is a bounded operator from $H^{-1}(\Rniip)$ to $H^{-1}_r(\Rniip)$, and so
\[
\|(J_{s} Q - Q J_{s})w\|_{H^{-1}_r(\Rniip)} \lesssim \|r(J_{s} Q - Q J_{s})w\|_{H^{-1}(\Rniip)} 
\]

Note that $J_s = h\partial_r + \frac{1}{r}T$, where  $T$ is a semiclassical pseudodifferential operator on $\Rn$ of order 1.  Meanwhile, $Q$ can be written as a combination of $\partial_r$ derivatives and differential operators on $\Rn$:
\[
Q = Ah^2\partial_r^2 + Bh\partial_r + C
\]
where $A,B,$ and $C$ are (perhaps r-dependent) differential operators of orders $0, 1,$ and $2$ respectively on $\Rn$ for each fixed $r$, with bounds uniform in $r$.  

If $w \in \S(\Rniip)$, then $Q w \in \S(\Rniip)$.  

Then
\[
\|r(J_{s} Q - Q J_{s})w\|_{H^{-1}(\Rniip)} = \|r[h \partial_r + \frac{1}{r}T, Ah^2\partial_r^2 + Bh\partial_r + C]w\|_{H^{-1}(\Rniip)}.
\]

Note that $T$ commutes with $\partial_r$.  Therefore we have
\begin{eqnarray*}
\|r(J_{s} Q - Q J_{s})w\|_{H^{-1}(\Rniip)} &\leq& \|r[h \partial_r, Q]w\|_{H^{-1}(\Rniip)} + \| [T,A]h^2\partial_r^2 w \|_{H^{-1}(\Rniip)} \\
                                          & & + \| \frac{h}{r}[T,A]h\partial_r w\|_{H^{-1}(\Rniip)} + \| \frac{2h^2}{r^2}[T,A]w\|_{H^{-1}(\Rniip)} \\
                                          & & + \|[T,B]h\partial_r w\|_{H^{-1}(\Rniip)} + \| \frac{h}{r} [T,B]w \|_{H^{-1}(\Rniip)} \\
                                          & & + \|[T,C] w\|_{H^{-1}(\Rniip)} \\
\end{eqnarray*}

Now $r[h\partial_r, Q] = hrE = hEr +h^2E'$, where $E$ and $E'$ are second and first order semiclassical differential operators, by the product rule.  Meanwhile, $[T,A] = hE_0, [T,B] = hE_1,$ and $[T,C] = hE_2$, where $E_0, E_1,$ and $E_2$ are semiclassical pseudodifferential operators on $\Rn$ of orders $0, 1,$ and $2$ respectively.  Therefore

\begin{eqnarray*}
& & \|r(J_{s} Q - Q J_{s})w\|_{H^{-1}(\Rniip)} \\
&\leq& \|hEr w\|_{H^{-1}(\Rniip)} + \|h^2 E' w\|_{H^{-1}(\Rniip)} + \| hE_0 h^2\partial_r^2 w \|_{H^{-1}(\Rniip)} \\
& & + \| \frac{h^2}{r}E_0 h\partial_r w\|_{H^{-1}(\Rniip)} + \| \frac{2h^3}{r^2}E_0 w\|_{H^{-1}(\Rniip)} \\
& & + \| hE_1 h\partial_r w\|_{H^{-1}(\Rniip)} + \| \frac{h}{r}E_1 w\|_{H^{-1}(\Rniip)} + \|hE_2 w\|_{H^{-1}(\Rniip)}\\
\end{eqnarray*}
 
Now $E$ is bounded from $H^1(\Rniip)$ to $H^{-1}(\Rniip)$, and $E'$ is bounded from $L^2(\Rniip)$ to $H^{-1}(\Rniip)$.  In addition, by the lemma on the boundedness of pseudodifferential operators on $\Rn$ in $\RnI$, $E_1^{*}$ is bounded from $H^1_0(\Rniip)$ to $L^2(\Rniip)$, so by duality, $E_1$ is bounded from $L^2(\Rniip)$ to $H^{-1}(\Rniip)$.  Meanwhile, if $G$ is an invertible semiclassical pseudodifferential operator of order $1$ on $\Rn$, then $G$ is bounded from $L^2(\Rniip)$ to $H^{-1}(\Rniip)$, and $G^{-1}E_2$ is bounded from $H^1(\Rniip)$ to $L^2(\Rniip)$, so $E_2$ is bounded from $H^1(\Rniip)$ to $H^{-1}(\Rniip)$.  Finally, $E_0^{*}$ is bounded from $H^1(\Rniip)$ to $H^1(\Rniip)$ and maps functions with trace 0 on the boundary of $\Rniip$ to other functions with trace 0 on that boundary, so it is bounded from $H^1_0(\Rniip)$ to $H^1_0(\Rniip)$. Therefore by duality, $E_0$ is bounded from $H^{-1}(\Rniip)$ to $H^{-1}(\Rniip)$.  Also, note that $\frac{1}{r} \leq 1$ on $\Rniip$. 

Therefore
\begin{eqnarray*}
& & \|(J_{s} Q - Q J_{s})w\|_{H^{-1}_r(\Rniip)} \\
&\lesssim& h\|rw\|_{H^{1}(\Rniip)} + h^2 \|w\|_{L^2(\Rniip)} + h C_{\d} \| h^2\partial_r^2 w \|_{H^{-1}(\Rniip)} \\
& & + h^2 C_{\d} \| h\partial_r w\|_{H^{-1}(\Rniip)} + h^3 C_{\d} \| w\|_{H^{-1}(\Rniip)} \\
& & + h C_{\d} \| h\partial_r w\|_{L^2(\Rniip)} + h^2 C_{\d} \| w\|_{L^2(\Rniip)} + h C_{\d} \| w\|_{H^{1}(\Rniip)} \\
&\lesssim& h C_{\d} \| r w \|_{H^1(\Rniip)} \\
\end{eqnarray*}
The $C_{\d}$ comes from the dependence of $T$ upon $\d$. 

\vspace{4mm}

ii) Now 
\begin{eqnarray*}
\| J_{s} \chi J^{-1}_{s}w \|_{L^2(\Rniip)} &=&    \| (h\partial_r + \frac{T}{r} )\chi J^{-1}_{s}w \|_{L^2(\Rniip)} \\
                                          &\geq& \| \chi (h\partial_r + \frac{T}{r} ) J^{-1}_{s}w \|_{L^2(\Rniip)} \\
                                          & &    - \| hE J^{-1}_{s}w \|_{L^2(\Rniip)} \\
\end{eqnarray*}
where for each fixed $r$, $E$ is an order zero pseudodifferential operator on $S^n$.  Therefore
\begin{eqnarray*}
\| J_{s} \chi J^{-1}_{s}w \|_{L^2(\Rniip)} &\geq& \| \chi J_s J^{-1}_{s}w \|_{L^2(\Rniip)} - h C_{\d} \| J^{-1}_{s}w \|_{L^2(\Rniip)} \\
                                          &\geq& \| \chi w \|_{L^2(\Rniip)} - h C_{\d} \| r w \|_{L^2(\Rniip)} \\
\end{eqnarray*}

\end{proof}

It would be nice if $\|J_{s}w\|_{H^{-1}_r(\Rniip)} \simeq \|w\|_{L^2(\Rniip)}$, but this is not quite true: since $J_s$ is not an isomorphism from $H^1_{0,r}(\Rniip)$ to $L^2(\Rniip)$, there is no reason to expect this to be true.  Instead, we have the following result.

\begin{lemma}\label{smallJHneg}
Suppose $u \in \S(\Rniip)$.  If $g$ is defined by 
\[
\hat{g}(r, \xi) = \frac{2\mathrm{Re}F_s(\xi)-h}{h}\int_1^{\infty} \hat{u}(t,\xi) r^{-\frac{F_s(\xi)}{h}} t^{-\frac{\overline{F_s(\xi)}}{h}}dt ,
\]
then
\[
\|J_{s}u\|_{H^{-1}_r(\Rniip)} \simeq \|u - g\|_{L^2(\Rniip)}
\] 
\end{lemma}

\begin{proof}
Suppose $u \in \S(\Rniip)$.  Define $g$ as above.  We have
\begin{eqnarray*}
& & \int_1^{\infty}|\hat{g}(r,\xi)|^2 dr \\
&=& \left| \frac{2\mathrm{Re}F_s-h}{h} \right|^2 \int_1^{\infty} \left| \int_1^{\infty}\hat{u}(t,\xi) r^{-\frac{F_s}{h}}t^{-\frac{\overline{F_s}}{h}}dt \right|^2 dr \\
&\leq& \left| \frac{2\mathrm{Re}F_s-h}{h} \right|^2 \int_1^{\infty} \int_1^{\infty} |\hat{u}(t,\xi)|^2 dt \int_1^{\infty} (rt)^{-2\frac{\mathrm{Re} F_s}{h}}dt \, dr \\
&=& \left| \frac{2\mathrm{Re}F_s-h}{h} \right|^2 \int_1^{\infty} |\hat{u}(t,\xi)|^2 dt  \int_1^{\infty} (t)^{-2\frac{\mathrm{Re} F_s}{h}} dt \int_1^{\infty} (r)^{-2\frac{\mathrm{Re} F_s}{h}} dr \\
&=& \left| \frac{2\mathrm{Re}F_s-h}{h} \right|^2 \int_1^{\infty} |\hat{u}(t,\xi)|^2 dt  \left( \frac{h}{2\mathrm{Re}F_s-h} \right) \left( \frac{h}{2\mathrm{Re} F_s-h} \right) \\
&=& \int_1^{\infty} |\hat{u}(t,\xi)|^2 dt \\
\end{eqnarray*}
Therefore $g$ makes sense as an element of $L^2(\Rniip)$, and $\|g\|_{L^2(\Rniip)} \leq \|u\|_{L^2(\Rniip)}$. Note that 
\[
\widehat{J_{s}g} = \left( \frac{F_s(\xi)}{r} + h\partial_r \right)\hat{g} = 0.
\]
Therefore
\begin{equation*}
\begin{split}
\|J_{s} u \|_{H^{-1}_r(\Rniip)} &= \sup_{w \in H^1_{0,r}(\Rniip), w \neq 0} \frac{|(J_{s}u,w)|}{\|w\|_{H^1_r(\Rniip)}} \\
                               &= \sup_{w \in H^1_{0,r}(\Rniip), w \neq 0} \frac{|(J_{s}(u-g),w)|}{\|w\|_{H^1_r(\Rniip)}} \\
                               &= \sup_{w \in H^1_{0,r}(\Rniip), w \neq 0} \frac{|(u-g,J_{s}^{*}w)|}{\|w\|_{H^1_r(\Rniip)}}.
\end{split}
\end{equation*}
Since $J_{s}^{*}: H^1_r(\Rniip) \rightarrow L^2(\Rniip)$ is an isomorphism, 

\begin{equation}\label{dualitycalc}
\|J_{s} u \|_{H^{-1}_r(\Rniip)} \simeq \sup_{w \in H^1_{0,r}(\Rniip), J_{s}^{*}w \neq 0} \frac{|(u-g,J_{s}^{*}w)|}{\|J_{s}^{*}w\|_{L^2(\Rniip)}}
\end{equation}
Now $J_{s}^{*}w \in L^2(\Rniip)$, so 
\[
\|J_{s} u \|_{H^{-1}_r(\Rniip)} \lesssim \|u - g\|_{L^2(\Rniip)}.
\]
On the other hand, note that $u - g = J_{s}^{*}J_{s}^{*-1}(u-g)$.  $J_{s}^{*-1}(u-g) \in H^1_r(\Rniip)$, and 
\begin{equation*}
\begin{split}
\widehat{J^{*-1}_{s}g}(1,\xi) &= h^{-1}\int_1^{\infty} h^{-1}(2\mathrm{Re}F_s-h) \left( \int_1^{\infty}\hat{u}(t,\xi)a^{-\frac{F_s}{h}}t^{-\frac{\overline{F_s}}{h}}dt \right) a^{-\frac{\overline{F_s}}{h}}da \\
                    &= h^{-1}\int_1^{\infty}\hat{u}(t,\xi)t^{-\frac{\overline{F_s}}{h}}dt (h^{-1})(2\mathrm{Re}F_s-h)\int_1^{\infty}  a^{-2\mathrm{Re}\frac{F_s}{h}}da \\
                    &= h^{-1}\int_1^{\infty}\hat{u}(t,\xi)t^{-\frac{\overline{F_s}}{h}}dt \\
                    &= \widehat{J^{*-1}_{s}u}(1,\xi) \\
\end{split}                             
\end{equation*}
Therefore $J_{s}^{*-1}(u-g) \in H^1_{r,0}(\Rniip)$.  If $u-g = 0$, then the lemma is true by \eqref{dualitycalc}.  Otherwise, we can pick $w = J_{s}^{*-1}(u-g)$ in \eqref{dualitycalc} to show that 
\[
\|J_{s} u \|_{H^{-1}_r(\Rniip)} \gtrsim \|u - g\|_{L^2(\Rniip)}.
\]
This finishes the proof.  

\end{proof}

\section{The Small Frequency Case}

Now we are ready to prove the small frequency case.  Suppose $\chi_2(r,\theta) \in C^{\infty}(\Rniip)$ is a cutoff function which is 1 on $\s(\Omt)$ and has support inside $\s(\Omt_2)$.

If $w \in \Czinf(\s(\Omt))$, then $w_s \in \S(\Rniip)$, supported away from $r = 1$.  Therefore $J^{-1}_{s}w_s \in \S(\Rniip)$, and is supported away from $r = 1$.  Then $\chi_2 J^{-1}_{s} w_s$ is in $\Czinf(\s(\Omt_2))$.  Therefore by \eqref{flatCarl},
\[
\frac{h}{\sqrt{\e}} \|\chi_2 J^{-1}_{s} w_s \|_{H^1(\s(\Omt_2))} \lesssim \|\Lphets \chi_2 J^{-1}_{s} w_s \|_{L^2(\s(\Omt_2))}.  
\]
Since $\chi_2 J^{-1}_{s} w_s \in \Czinf(\s(\Omt_2))$, the $H^1$ and $H^1_r$ norms are comparable, so 
\[
\frac{h}{\sqrt{\e}} \|\chi_2 J^{-1}_{s} w_s \|_{H^1_r(\Rniip)} \lesssim \|\Lphets \chi_2 J^{-1}_{s} w_s \|_{L^2(\Rniip)}.  
\]
Using the boundedness properties from Lemma \ref{smallJbounds}, 
\[
\frac{h}{\sqrt{\e}} \|J_{s} \chi_2 J^{-1}_{s} w_s \|_{L^2(\Rniip)} \lesssim \|\Lphets \chi_2 J^{-1}_{s} w_s \|_{L^2(\Rniip)} 
\]
so applying the second part of Lemma \ref{smallJprops}
\[
\frac{h}{\sqrt{\e}}\|\chi_2 w_s\|_{L^2(\Rniip)} \lesssim \|\Lphets \chi_2 J^{-1}_{s}w_s \|_{L^2(\Rniip)} + C_{\d}\frac{h^2}{\e}\|r w_s \|_{L^2(\Rniip)}.
\]
Now $\chi_2 w_s = \chi_2 P w$.  Since $w$ is only supported on the region where $\chi_2$ is identically 1, 
\[
\chi_2 w_s = Pw + O(h^{\infty}) E w = w_s + O(h^{\infty}) E w
\]
where $E$ is an pseudodifferential operator of order $0$ (actually a smoothing operator) on $\Rn$.  Therefore
\[
\frac{h}{\sqrt{\e}}\|\chi_2 w_s\|_{L^2(\Rniip)} \gtrsim \frac{h}{\sqrt{\e}}\|w_s\|_{L^2(\Rniip)} - O(h^{\infty})\|w\|_{L^2(\Rniip)},
\]
and so
\[
\frac{h}{\sqrt{\e}}\|w_s\|_{L^2(\Rniip)} \lesssim \|\Lphets \chi_2 J^{-1}_{s}w_s \|_{L^2(\Rniip)} +C_{\d}\frac{h^2}{\e}\| r w_s \|_{L^2(\Rniip)} + O(h^{\infty})\|w\|_{L^2(\Rniip)}.
\]
For small enough $h$, the second last term can be absorbed into the left side ($r$ is bounded on the support of $w_s$) to give 
\[
\frac{h}{\sqrt{\e}} \|w_s \|_{L^2(\Rniip)} \lesssim \|\Lphets \chi_2 J^{-1}_{s} w_s \|_{L^2(\s(\Omt_2))} + O(h^{\infty})\|w\|_{L^2(\Rniip)}.
\]
By the product rule,
\[
\frac{h}{\sqrt{\e}} \|w_s \|_{L^2(\Rniip)} \lesssim \|\chi_2 \Lphets J^{-1}_{s} w_s \|_{L^2(\s(\Omt_2))} + h\|J^{-1}_{s} w_s \|_{H^1(\s(\Omt_2))} + O(h^{\infty})\|w\|_{L^2(\Rniip)}.
\]
On $\s(\Omt_2)$, the $H^1$ and $H^1_r$ norms are comparable, so
\[
\frac{h}{\sqrt{\e}} \|w_s \|_{L^2(\Rniip)} \lesssim \|\Lphets J^{-1}_{s} w_s \|_{L^2(\Rniip)} + h\|J^{-1}_{s} w_s \|_{H^1_r(\Rniip))} + O(h^{\infty})\|w\|_{L^2(\Rniip)}.
\]
Using the boundedness properties again, 
\[
\frac{h}{\sqrt{\e}} \|w_s \|_{L^2(\Rniip)} \lesssim \|\Lphets J^{-1}_{s} w_s \|_{L^2(\Rniip)} + h\| w_s \|_{L^2(\Rniip)} +O(h^{\infty})\|w\|_{L^2(\Rniip)}.
\]
The second last term can be absorbed into the left side to give

\begin{equation}\label{firstpart}
\frac{h}{\sqrt{\e}} \|w_s \|_{L^2(\Rniip)} \lesssim \|\Lphets J^{-1}_{s} w_s \|_{L^2(\Rniip)} + O(h^{\infty})\|w\|_{L^2(\Rniip)}.
\end{equation}

I want to combine this last inequality with Lemma \ref{smallJHneg} to get 
\[
\frac{h}{\sqrt{\e}} \|w_s \|_{L^2(\Rniip)} \lesssim \|J_{s}\Lphets J^{-1}_{s} w_s \|_{H^{-1}_r(\Rniip)} + O(h^{\infty})\|w \|_{L^2(\Rniip)}.
\]
To do this I need to show that if $u = \Lphets J^{-1}_{s} w_s$, then the function $g$ defined in Lemma \ref{smallJHneg} satisfies a bound like
\[
\|g\|_{L^2(\Rniip)} \leq \half \|u\|_{L^2(\Rniip)}.
\] 
This is not true in general, but happens in this case because of the particular form of $u$.  Let $v = J^{-1}_s w_s$, and consider $\hat{g}(r,\xi)$.  
\begin{eqnarray*}
\hat{g} &=& \frac{2\mathrm{Re}F_s-h}{h}\int_1^{\infty} \widehat{\Lphets v}(t,\xi) r^{-\frac{F_s}{h}} t^{-\frac{\overline{F_s}}{h}}dt \\
    &=& \frac{2\mathrm{Re}F_s-h}{h}\int_1^{\infty} h^2\partial_t^2 \hat{v}(t,\xi) r^{-\frac{F_s}{h}} t^{-\frac{\overline{F_s}}{h}}dt \\
    & & + \frac{2\mathrm{Re}F_s-h}{h}\int_1^{\infty} \mathcal{F}(|\g_f|^2 h^2\partial_t^2 v)(t,\xi) r^{-\frac{F_s}{h}} t^{-\frac{\overline{F_s}}{h}}dt \\
    & & - \frac{2\mathrm{Re}F_s-h}{h}\int_1^{\infty} \frac{2}{t} h\partial_t \hat{v}(t,\xi) r^{-\frac{F_s}{h}} t^{-\frac{\overline{F_s}}{h}}dt \\
    & & - \frac{2\mathrm{Re}F_s-h}{h}\int_1^{\infty} \frac{2}{t}\mathcal{F}\left( \frac{h}{\e}\log(tf(\th))h\partial_t v \right)(t,\xi) r^{-\frac{F_s}{h}} t^{-\frac{\overline{F_s}}{h}} dt \\
    & & - \frac{2\mathrm{Re}F_s-h}{h}\int_1^{\infty} \frac{2}{t}\mathcal{F}\left( \b_f \cdot h\grad_{\th}h\partial_t v \right)(t,\xi) r^{-\frac{F_s}{h}} t^{-\frac{\overline{F_s}}{h}} dt \\
    & & + \frac{2\mathrm{Re}F_s-h}{h}\int_1^{\infty} \frac{1}{t^2}\mathcal{F}((1+h^2 L_{S^n})v)(t,\xi) r^{-\frac{F_s}{h}} t^{-\frac{\overline{F_s}}{h}}dt \\
    & & + \frac{2\mathrm{Re}F_s-h}{h}\int_1^{\infty} \frac{1}{t^2} \mathcal{F}\left( \left( (2\frac{h}{\e}\log(tf(\th))+\frac{h^2}{\e^2}\log^2(tf(\th)) \right) v \right)(t,\xi) r^{-\frac{F_s}{h}} t^{-\frac{\overline{F_s}}{h}}dt. \\
\end{eqnarray*}
Here $\mathcal{F}$ represents the same thing the hat $\hat{ }$ does in the case where LaTeX's wide hat looks too strange to be used.  

Rewriting, we have
\begin{eqnarray*}        
\hat{g} &=& \frac{2\mathrm{Re}F_s-h}{h}\int_1^{\infty} (1 + K^2) h^2\partial_t^2 \hat{v}(t,\xi) r^{-\frac{F_s}{h}} t^{-\frac{\overline{F_s}}{h}}dt \\
    & & - \frac{2\mathrm{Re}F_s-h}{h}\int_1^{\infty} \frac{2}{t}(1+iK\xi_n)h\partial_t \hat{v}(t,\xi) r^{-\frac{F_s}{h}} t^{-\frac{\overline{F_s}}{h}}dt \\
    & & + \frac{2\mathrm{Re}F_s-h}{h}\int_1^{\infty} \frac{1}{t^2}(1-|\xi|^2)\hat{v}(t,\xi) r^{-\frac{F_s}{h}} t^{-\frac{\overline{F_s}}{h}}dt \\
    & & + \frac{2\mathrm{Re}F_s-h}{h}\int_1^{\infty} \mathcal{F}((|\g_f|^2-K^2) h^2\partial_t^2 v)(t,\xi) r^{-\frac{F_s}{h}} t^{-\frac{\overline{F_s}}{h}}dt \\
    & & -\frac{2\mathrm{Re}F_s-h}{h}\int_1^{\infty} \mathcal{F} \left( \frac{2}{t}\frac{h}{\e}\log(tf(\th))h\partial_t v \right)(t,\xi) r^{-\frac{F_s}{h}} t^{-\frac{\overline{F_s}}{h}} dt \\
    & & - \frac{2\mathrm{Re}F_s-h}{h}\int_1^{\infty} \mathcal{F}\left( \frac{2}{t}(\b_f-Ke_n)\cdot h\grad_{\th}h\partial_t v \right)(t,\xi) r^{-\frac{F_s}{h}} t^{-\frac{\overline{F_s}}{h}}dt \\
    & & + \frac{2\mathrm{Re}F_s-h}{h}\int_1^{\infty} \mathcal{F} \left( \frac{1}{t^2} \left( (2\frac{h}{\e}\log(tf(\th))+\frac{h^2}{\e^2}\log^2(tf(\th)) \right) v \right)(t,\xi) r^{-\frac{F_s}{h}} t^{-\frac{\overline{F_s}}{h}}dt. \\
    & & + \frac{2\mathrm{Re}F_s-h}{h}\int_1^{\infty} \frac{1}{t^2}\mathcal{F}(h^2(L_{S^n} -\Lap_{\th})v)(t,\xi) r^{-\frac{F_s}{h}} t^{-\frac{\overline{F_s}}{h}}dt \\
\end{eqnarray*}
Integrating by parts in the first term gives
\begin{eqnarray*}
& & \frac{2\mathrm{Re}F_s-h}{h}\int_1^{\infty} (1 + K^2) h^2\partial_t^2 \hat{v}(t,\xi) r^{-\frac{F_s}{h}} t^{-\frac{\overline{F_s}}{h}}dt \\
&=& \frac{2\mathrm{Re}F_s-h}{h}\int_1^{\infty} \frac{\overline{F_s}}{t}(1 + K^2) h\partial_t \hat{v}(t,\xi) r^{-\frac{F_s}{h}} t^{-\frac{\overline{F_s}}{h}}dt \\
&=& \frac{2\mathrm{Re}F_s-h}{h}\int_1^{\infty} \left( \frac{\overline{F_s}}{t} \right)^2(1 + K^2) \hat{v}(t,\xi) r^{-\frac{F_s}{h}} t^{-\frac{\overline{F_s}}{h}}dt \\
& & +\frac{2\mathrm{Re}F_s-h}{h}\int_1^{\infty} h\frac{\overline{F_s}}{t^2}(1 + K^2) \hat{v}(t,\xi) r^{-\frac{F_s}{h}} t^{-\frac{\overline{F_s}}{h}}dt. \\
\end{eqnarray*}
There are no boundary terms from the integration by parts, because $w$ is supported away from $r = 1$, and hence $w_s$ and $v$ are as well.  Integrating by parts in the second term gives
\begin{eqnarray*}
& & \frac{2\mathrm{Re}F_s-h}{h}\int_1^{\infty} \frac{2}{t}(1+iK\xi_n)h\partial_t \hat{v}(t,\xi) r^{-\frac{F_s}{h}} t^{-\frac{\overline{F_s}}{h}}dt \\
&=& \frac{2\mathrm{Re}F_s-h}{h}\int_1^{\infty} \frac{2\overline{F_{sk}}}{t^2}( 1 + iK\xi_n) \hat{v}(t,\xi) r^{-\frac{F_s}{h}} t^{-\frac{\overline{F_s}}{h}}dt \\ 
& & + \frac{2\mathrm{Re}F_s-h}{h}\int_1^{\infty} \frac{2h}{t^2}( 1 + iK\xi_n) \hat{v}(t,\xi) r^{-\frac{F_s}{h}} t^{-\frac{\overline{F_s}}{h}}dt \\ 
\end{eqnarray*}
Therefore we have 
\begin{eqnarray*}        
\hat{g} &=& \frac{2\mathrm{Re}F_s-h}{h}\int_1^{\infty} \frac{1}{t^2}((1+K^2)\overline{F_s}^2 -2(1+iK\xi_n)\overline{F_s}+1-|\xi|^2)\hat{v}(t,\xi) r^{-\frac{F_s}{h}} t^{-\frac{\overline{F_s}}{h}}dt \\
    & & +\frac{2\mathrm{Re}F_s-h}{h}\int_1^{\infty} h\frac{\overline{F_s}}{t^2}(1 + K^2) \hat{v}(t,\xi) r^{-\frac{F_s}{h}} t^{-\frac{\overline{F_s}}{h}}dt. \\
    & & -\frac{2\mathrm{Re}F_s-h}{h}\int_1^{\infty} \frac{2h}{t^2}(1+iK\xi_n) \hat{v}(t,\xi) r^{-\frac{F_s}{h}} t^{-\frac{\overline{F_s}}{h}}dt \\ 
    & & + \frac{2\mathrm{Re}F_s-h}{h}\int_1^{\infty} \mathcal{F}((|\g_f|^2-K^2) h^2\partial_t^2 v)(t,\xi) r^{-\frac{F_s}{h}} t^{-\frac{\overline{F_s}}{h}}dt \\
    & & -\frac{2\mathrm{Re}F_s-h}{h}\int_1^{\infty} \mathcal{F} \left( \frac{2}{t}\frac{h}{\e}\log(tf(\th))h\partial_t v \right)(t,\xi) r^{-\frac{F_s}{h}} t^{-\frac{\overline{F_s}}{h}} dt \\
    & & - \frac{2\mathrm{Re}F_s-h}{h}\int_1^{\infty} \mathcal{F}\left( \frac{2}{t}(\b_f-Ke_n)\cdot h\grad_{\th}h\partial_t v \right)(t,\xi) r^{-\frac{F_s}{h}} t^{-\frac{\overline{F_s}}{h}}dt \\
    & & + \frac{2\mathrm{Re}F_s-h}{h}\int_1^{\infty} \mathcal{F} \left( \frac{1}{t^2} \left( (2\frac{h}{\e}\log(tf(\th))+\frac{h^2}{\e^2}\log^2(tf(\th)) \right) v \right)(t,\xi) r^{-\frac{F_s}{h}} t^{-\frac{\overline{F_s}}{h}}dt. \\
    & & + \frac{2\mathrm{Re}F_s-h}{h}\int_1^{\infty} \frac{1}{t^2}\mathcal{F}(h^2(L_{S^n} -\Lap_{\th})v)(t,\xi) r^{-\frac{F_s}{h}} t^{-\frac{\overline{F_s}}{h}}dt \\
\end{eqnarray*}
Applying the same reasoning as in Lemma \ref{smallJHneg}, 
\begin{eqnarray*}
\|g\|^2_{L^2(\Rniip)} &\leq&  h^{-n} \left\| \frac{1}{r^2}((1+K^2)\overline{F_s}^2 -2(1+iK\xi_n)\overline{F_s}+1-|\xi|^2) \hat{v}(r,\xi) \right\|^2_{L^2(\Rniip)} \\
                     & & + h^{-n} \left\| h\frac{\overline{F_s}}{r^2}(1 + K^2) \hat{v}(r,\xi) \right\|^2_{L^2(\Rniip)}\\
                     & & +  \left\| \frac{2h}{r^2}( 1+Kh\partial_{\th_n}) v(r,\th) \right\|^2_{L^2(\Rniip)} \\
                     & & + \|(|\g_f|^2-K^2) h^2\partial_r^2 v\|^2_{L^2(\Rniip)} + \left\| \frac{2}{r}\frac{h}{\e}\log(rf(\th))h\partial_r v \right\|^2_{L^2(\Rniip)} \\
                     & & + \left\| \frac{2}{r}(\b_f - Ke_n)\cdot h\grad_{\th}h\partial_r v \right\|^2_{L^2(\Rniip)} \\
                     & & + \left\| \frac{1}{r^2}\left( 2\frac{h}{\e}\log(rf(\th))+\frac{h^2}{\e^2}\log^2(rf(\th)) \right) v \right\|^2_{L^2(\Rniip)} \\
                     & & + \|h^2\frac{1}{r^2}(L_{S^n} -\Lap_{\th})v\|^2_{L^2(\Rniip)} \\
                     &\lesssim& h^{-n}\| ((1+K^2)\overline{F_s}^2 -2(1+iK\xi_n)\overline{F_s} + 1 -|\xi|^2) \hat{v}(r,\xi)\|^2_{L^2(\Rniip)} \\
                     & & + h^{-n} \| h\overline{F_s}(1 + K^2) \hat{v}(r,\xi)\|^2_{L^2(\Rniip)}\\
                     & & + h^2 \left\| v \right\|^2_{H^1(\Rniip)} + C_{\mu}^2 \|h^2\partial_r^2 v\|^2_{L^2(\Rniip)}\\
                     & & + \frac{h^2}{\e^2}\| h\partial_r v\|^2_{L^2(\Rniip)} + C_{\mu}^2 \| h\grad_{\th}h\partial_r v\|^2_{L^2(\Rniip)} \\
                     & & + \frac{h^2}{\e^2}\| v \|^2_{L^2(\Rniip)} + C_{\mu}^2 \|v\|_{H^2(\Rniip)}^2 \\
\end{eqnarray*}
where $C_{\mu}$ goes to zero as $\mu$ does.

Now $F_s(\xi)$ is designed so that $\overline{F_s(\xi)}$ is very nearly a solution to $(1 + K^2)X^2 -2(1+iK\xi_n)X + 1 - |\xi|^2 = 0$ when $\hat{w}_s \neq 0$ and hence when $\hat{v} \neq 0$.  More precisely, 
\[
|(1 + K^2)\overline{F_s(\xi)}^2 -2(1+iK\xi_n)\overline{F_s(\xi)} + 1 -|\xi|^2| \lesssim \d (|F_s(\xi)| + |\xi_n|) \lesssim \d|F_s(\xi)|, 
\]
so
\begin{eqnarray*}
\|g\|^2_{L^2(\Rniip)} &\lesssim& h^{-n} \| \d|F_s(\xi)| \hat{v}(r,\xi)\|^2_{L^2(\Rniip)}\\
                     & &        + h^{-n} \| h\overline{F_s}(1 + K^2) \hat{v}(r,\xi) \|^2_{L^2(\Rniip)} \\
                     & & + C_{\mu}^2 \|v\|_{H^2(\Rniip)}^2 \\
                     &\lesssim& \d^2 \| v \|^2_{H^1(\Rniip)} + h^2 \| v \|^2_{H^1(\Rniip)} + C_{\mu}^2 \|v\|_{H^2(\Rniip)}^2 \\
                     &\lesssim& (\d^2 + C_{\mu}^2) \| v\|^2_{H^2(\Rniip)} \\
\end{eqnarray*}

This gives an estimate for $g$ in terms of $v$.  However, we want the estimate to be in terms of $u$.  We have $u = \Lphets v$, so 
\begin{eqnarray*}
\|u\|^2_{L^2(\Rniip)} &=& \|\Lphets v \|^2_{L^2(\Rniip)} \\
                      &\geq& \left\| ((1 + K^2)h^2 \partial_r^2 -\frac{2}{r}(1+Kh\partial_{\th_n})h\partial_r +\frac{1}{r^2}(1 +h^2\Lap_{\th}) v \right\|^2_{L^2(\Rniip)} \\
                     & & - \| (|\g_f|^2 - K^2)h^2\partial_r^2 v \|^2_{L^2(\Rniip)} \\
                     & & - \left\| \frac{2}{r}\frac{h}{\e}\log(rf(\th))h\partial_r v \right\|^2_{L^2(\Rniip)} \\
                     & & - \left\| \frac{2}{r}(\b_f - Ke_n)\cdot h\grad_{\th}h\partial_r v \right\|^2_{L^2(\Rniip)} \\
                     & & - \left\| \frac{1}{r^2}((2\frac{h}{\e}\log(rf(\th))+\frac{h^2}{\e^2}\log^2(rf(\th))) v \right\|^2_{L^2(\Rniip)} \\
                     & & - \left\| \frac{h^2}{r^2}(L_{S^n} - \Lap_{\th})v \right\|^2_{L^2(\Rniip)} \\
           &\gtrsim& \left\|\left( (1+K^2)h^2 \partial_r^2 -\frac{2}{r}(1+Kh\partial_{\th_n})h\partial_r +\frac{1}{r^2}(1 +h^2\Lap_{\th}\right)v \right\|^2_{L^2(\Rniip)} \\
                     & & - C_{\mu}^2 \| v \|^2_{H^2(\Rniip)} \\
\end{eqnarray*}
Writing the last expression in terms of $\hat{v}$, we get
\[                    
h^{-n}\left\| \left((1+K^2)h^2 \partial_r^2 -\frac{2}{r}(1+iK\xi_n)h\partial_r +\frac{1}{r^2}(1-|\xi|^2) \right)\hat{v}(r,\xi) \right\|^2_{L^2(\Rniip)} - C_{\mu}^2 \| v \|^2_{H^2(\Rniip)}
\]

Now $\hat{v}(r,\xi) = \mathcal{F}(J^{-1}_s P w)(r,\xi)$ is only non-zero for $\xi$ such that $|\xi|^2 \leq \half + \half \frac{|K|^2}{1+|K|^2} < 1$.  The operator 
\[
(1 + K^2)h^2 \partial_r^2 - \frac{2}{r}(1+iK\xi_n)h\partial_r +\frac{1}{r^2}(1-|\xi|^2)
\]
coincides, for $r > 1$, with a differential operator in $r$ of the form
\[
(1 + K^2)h^2 \partial_r^2 - 2\omega(1+iK\xi_n)h\partial_r +\omega^2(1-|\xi|^2)
\]
where $\omega$ is a smooth function that coincides with $\frac{1}{r}$ for $r>1$.  This is second order elliptic for each $|\xi|$ such that $\hat{v}(r,\xi)$ is nonzero, and its symbol (in $r$) is bounded below, therefore
\begin{eqnarray*}
& & \int_1^{\infty} |((1+K^2)h^2 \partial_r^2 - \frac{2}{r}(1+iK\xi_n)h\partial_r +\frac{1}{r^2}(1-|\xi|^2))\hat{v}(r,\xi)|^2 dr \\
&\simeq& \int_1^{\infty} |(1 - h^2\partial_r^2)\hat{v}(r,\xi) |^2 dr \\
&\simeq& \int_1^{\infty} |(1 - h^2\partial_r^2 + |\xi|^2)\hat{v}(r,\xi)|^2 dr \\
&\simeq& \int_1^{\infty} |\mathcal{F}((1 - h^2\Lap)v)(r,\xi)|^2 dr.
\end{eqnarray*}
Then
\begin{eqnarray*}
& &  h^{-n}\|((1 + K^2)h^2 \partial_r^2 - \frac{2}{r}(1+iK\xi_n)h\partial_r +\frac{1}{r^2}(1-|\xi|^2))\hat{v}(r,\xi)\|_{L^2(\Rniip)}^2  \\
&\simeq& \| v \|^2_{H^2(\Rniip)}
\end{eqnarray*}
and so
\begin{eqnarray*}
\|u\|^2_{L^2(\Rniip)} &\gtrsim& \|v\|^2_{H^2(\Rniip)} - C_{\mu}^2 \| v \|^2_{H^2(\Rniip)} \\
                     &\gtrsim& \|v\|^2_{H^2(\Rniip)} \\
\end{eqnarray*}
for $\mu$ small enough.  

Plugging this into the inequality for $g$ gives
\[
\|g\|^2_{L^2(\Rniip)} \lesssim (\d^2 + C_{\mu}^2) \| u\|^2_{L^2(\Rniip)}.
\]
Taking $\mu$ and $\d$ small enough means
\[
\|g\|^2_{L^2(\Rniip)} \leq \half \| u\|^2_{L^2(\Rniip)}
\]

Combining this with \eqref{firstpart} now gives
\[
\frac{h}{\sqrt{\e}} \|w_s \|_{L^2(\Rniip)} \lesssim  \|J_{s}\Lphets J^{-1}_{s} w_s \|_{H^{-1}_r(\Rniip)}+ O(h^{\infty})\|w \|_{L^2(\Rniip)}.
\]
Now using the first part of Lemma \ref{smallJprops} gives
\begin{eqnarray*}
\frac{h}{\sqrt{\e}}\|w_s\|_{L^2(\Rniip)} &\lesssim& \|\Lphets J_{s} J^{-1}_{s} w_s \|_{H^{-1}_r(\Rniip)}+ C_{\d}h\|r J^{-1}_{s} w_s\|_{H^1(\Rniip)} \\
                                         & &        + O(h^{\infty})\|w \|_{L^2(\Rniip)}\\
                                         &\lesssim& \|\Lphets w_s \|_{H^{-1}_r(\Rniip)}+ C_{\d}h\|r J^{-1}_{s} w_s\|_{H^1(\Rniip)} \\
                                         & &        + O(h^{\infty})\|w \|_{L^2(\Rniip)}\\
\end{eqnarray*}
$\Lphets w_s$ is supported in the $r$ direction only for those $r$ which can come from $\Omt_2$, since $w_s$ is.  Therefore the $H^{-1}_r$ and $H^{-1}$ norms are comparable, and so
\begin{eqnarray*}
\frac{h}{\sqrt{\e}}\|w_s\|_{L^2(\Rniip)} &\lesssim& \|\Lphets w_s \|_{H^{-1}(\Rniip)}+ C_{\d}h\|r J^{-1}_{s} w_s\|_{H^1(\Rniip)} \\
                                         & &        + O(h^{\infty})\|w \|_{L^2(\Rniip)}\\
\end{eqnarray*}

Meanwhile, 
\[
\widehat{J^{-1}_{s} w_s}(r,\xi) = \frac{1}{h} \int_1^r \hat{w}_s(t,\xi) \left( \frac{t}{r} \right)^{\frac{F_s(\xi)}{h}} dt,
\]
and $\hat{w}_s(t,\xi)$ is supported only for $1 \leq t \leq C$ for some $C$ depending on $\s(\Omt_2)$.  Therefore for $r > 4C$, 
\[
|\widehat{J^{-1}_{s}w_s}(r,\xi)| \leq \left| \frac{1}{h}\int_1^C \hat{w}_s(t,\xi) \left( \frac{t}{2C}\right)^{\frac{F_s}{h}}dt\right| \left| \half \right|^{\mathrm{Re} \frac{F_s}{h}} \left| \frac{4C}{r} \right|^{\mathrm{Re} \frac{F_s}{h}},
\]
so
\[
|\widehat{J^{-1}_{s}w_s}(r,\xi)|^2 \lesssim \int_1^C |\hat{w}(t,\xi)|^2 dt \left| \half \right|^{\mathrm{Re} \frac{2F_s}{h}} \left| \frac{4C}{r} \right|^{\mathrm{Re} \frac{2F_s}{h}}
\]
Therefore
\[
\|rJ_s^{-1}w_s\|_{L^2(\Rniip)} \lesssim \|r J_s^{-1}w_s\|_{L^2(1 < r < 4C)} + O(h^{\infty})\|w_s\|_{L^2(\Rniip)}.
\]
Similar calculations for derivatives of $J_s^{-1}w$ give
\[
\|rJ_s^{-1}w_s\|_{H^1(\Rniip)} \lesssim \|r J_s^{-1}w_s\|_{H^1(1 < r < 4C)} + O(h^{\infty})\|w_s\|_{L^2(\Rniip)},
\]
so
\[
\|r J_s^{-1}w_s\|_{H^1(\Rniip)} \lesssim \|J_s^{-1}w_s\|_{H^1_r(\Rniip)} + O(h^{\infty})\|w_s\|_{L^2(\Rniip)}.
\]

Therefore
\begin{eqnarray*}
\frac{h}{\sqrt{\e}}\|w_s\|_{L^2(\Rniip)} &\lesssim& \|\Lphets w_s \|_{H^{-1}(\Rniip)}+ C_{\d}h\|J^{-1}_{s} w_s\|_{H^1_r(\Rniip)} \\
                                         & &        + O(h^{\infty})\|w \|_{L^2(\Rniip)}\\
\end{eqnarray*}
Applying the boundedness results gives
\begin{eqnarray*}
\frac{h}{\sqrt{\e}}\|w_s\|_{L^2(\Rniip)} &\lesssim& \|\Lphets w_s \|_{H^{-1}(\Rniip)}+ C_{\d}h\|w_s\|_{L^2(\Rniip)} \\
                                         & &        + O(h^{\infty})\|w \|_{L^2(\Rniip)}.\\
\end{eqnarray*}
For small enough $\e$, the second last term can be absorbed into the left side to give 
\[
\frac{h}{\sqrt{\e}}\|w_s\|_{L^2(\Rniip)} \lesssim \|\Lphets w_s \|_{H^{-1}(\Rniip)}+ O(h^{\infty})\|w \|_{L^2(\Rniip)}.
\]
This finishes the proof of Lemma \ref{smallCarl}.

\section{The Large Frequency Case}

Consider again the function $F:\R^n \rightarrow \mathbb{C}$ given by
\[
\overline{F(\xi)} = \frac{1}{1+K^2}\left( 1+iK\xi_n + \sqrt{2iK\xi_n-(K\xi_n)^2 +(1+K^2)|\xi|^2 -|K|^2} \right),
\]
but this time take the branch of the square root which has nonnegative real part.  Now $F$ is smooth except where
\[
\tau_{K}(\xi) = 2iK\xi_n-(K\xi_n)^2 +(1 + |K|^2)|\xi|^2 -|K|^2
\]
lies on the nonpositive real axis.  This happens when $\xi_n =0$ and 
\[
|\xi|^2 \leq \frac{|K|^2}{1 + |K|^2}.
\]
Therefore on the support of $1-\rho(\xi)$, $F$ is smooth.  Moreover, since the real part of the square root is nonnegative, both $|F|$ and the real part of $F$ and are bounded below by $\frac{1}{1+K^2}$.  Therefore we can pick a smooth function $F_{\ell}$ such that $F_{\ell}(\xi) = F(\xi)$ on the support of $1-\rho(\xi)$, and $\mathrm{Re}F_{\ell}, |F_{\ell}| > \half \frac{1}{1+K^2}$.  Note that for large $|\xi|$, we have $\mathrm{Re}F(\xi), |F(\xi)| \simeq 1 + |\xi|$, $F_{\ell}$ can satisfy
\[
\mathrm{Re}F_{\ell}(\xi), |F_{\ell}(\xi)| \simeq 1 + |\xi|
\]

In fact, if $\frac{K^2}{1+K^2} < r_0 < r_1$ and $0 < \d_0 < \d_1$, we can arrange for $F_{\ell} = F$ and $F_{\ell}$ to be smooth for $|\xi|^2 \geq r_0$ and $\xi_n \geq \d_0$.  

\[
\widehat{J_{\ell} u}(r, \xi) = \left( \frac{F_{\ell}(\xi)}{r} + h\partial_r \right) \hat{u}(r, \xi).
\]
This operator has adjoint $J_{s}^{*}$ given by
\[
\widehat{J_{\ell}^{*} u}(r, \xi) = \left( \frac{\overline{F_{\ell}(\xi)}}{r} - h\partial_r \right) \hat{u}(r, \xi).
\]
These operators have (right) inverses defined by 

\[
\widehat{J_{\ell}^{-1}u} (r,\xi) = h^{-1} \int_1^{r} \hat{u}(t,\xi) \left( \frac{t}{r} \right)^{\frac{F_{\ell}(\xi)}{h}} dt 
\]
and
\[
\widehat{J_{\ell}^{*-1}u} (r,\xi) = h^{-1} \int_r^{\infty} \hat{u}(t,\xi) \left( \frac{r}{t} \right)^{\frac{\overline{F_{\ell}(\xi)}}{h}} dt.
\]
Each of these is well defined for functions in $\S(\Rniip)$.

We have the following lemmas.
\begin{lemma}\label{largeJbounds}
$J_{\ell}, J_{\ell}^{*},J_{\ell}^{-1},$ and $J_{\ell}^{*-1}$ extend as bounded maps
\[
J_{\ell}, J_{\ell}^{*}: H^1_r(\Rniip) \rightarrow L^2(\Rniip)
\]
and
\[
J_{\ell}^{-1}, J_{\ell}^{*-1}: L^2(\Rniip) \rightarrow H^1_r(\Rniip).
\]
Moreover, the extensions of $J_{\ell}^{*}$ and $J^{*-1}_{\ell}$ are isomorphisms.  
\end{lemma}

\begin{lemma}\label{largeJprops}

i) Suppose $w \in \S(\Rniip) $.  Then if $Q$ is a second order semiclassical differential operator with bounded coefficients, then
\[
\|(J_{\ell} Q - Q J_{\ell})w\|_{H^{-1}_r(\Rniip)} \lesssim hC_{\d}\|rw\|_{H^1(\Rniip)}
\] 

ii) Suppose $w \in \S(\Rniip) $.  Let $\chi \in \S(\Rniip)$.  Then
\[
\| J_{\ell} \chi J^{-1}_{\ell}w \|_{L^2(\Rniip)} \gtrsim \|\chi w\|_{L^2(\Rniip)} - hC_{\d}\| rw\|_{L^2(\Rniip)}
\]

\end{lemma}

\begin{lemma}\label{largeJHneg}

Suppose $u \in \S(\Rniip)$.  If $g$ is defined by 
\[
\hat{g}(r, \xi) = \frac{2\mathrm{Re}F_{\ell}(\xi)-h}{h}\int_1^{\infty} \hat{u}(t,\xi) r^{-\frac{F_{\ell}(\xi)}{h}} t^{-\frac{\overline{F_{\ell}(\xi)}}{h}}dt ,
\]
then
\[
\|J_{\ell}u\|_{H^{-1}_r(\Rniip)} \simeq \|u - g\|_{L^2(\Rniip)}
\] 
\end{lemma}

The proofs of these lemmas are identical to the proofs of the equivalent lemmas in the small frequency case.

Now consider the Carleman estimate \eqref{flatCarl}.  By a similar argument as in the small frequency case, we get

\begin{equation}\label{largefirstpart}
\frac{h}{\sqrt{\e}} \|w_{\ell} \|_{L^2(\Rniip)} \lesssim \|\Lphets J^{-1}_{\ell} w_{\ell} \|_{L^2(\Rniip)}+ O(h^{\infty})\|w \|_{L^2(\Rniip)}.
\end{equation}

\noindent Again I want to combine this last inequality with Lemma \ref{largeJHneg} to get 
\[
\frac{h}{\sqrt{\e}} \|w_{\ell} \|_{L^2(\Rniip)} \lesssim \|J_{\ell}\Lphets J^{-1}_{\ell} w_{\ell} \|_{H^{-1}_r(\Rniip)} + O(h^{\infty})\|w \|_{L^2(\Rniip)}.
\]
To do this I need to show that if $u$ is of the form $u = \Lphets J^{-1}_{\ell} w_{\ell}$, then the function $g$ defined in Lemma \ref{largeJHneg} satisfies a bound like
\[
\|g\|_{L^2(\Rniip)} \leq \half \|u\|_{L^2(\Rniip)} + O(h)\|w_{\ell} \|_{L^2(\Rniip)}.
\] 
The approach used in the small frequency case does not work here, because $\Lphets$ is not at all elliptic from the point of view of $w_{\ell}$.  However, now $\Lphets$ can be factored into a composition of two operators, one of which has the desired properties.  

Let $\zeta(\xi)$ be a smooth cutoff function which is identically 1 on the set where $|\xi|^2 \geq r_1$ or $|\xi_n| \geq \d_1$, and vanishes if $|\xi|^2 \leq r_0$ or $|\xi_n| \leq \d_0$.  Let 
\[
G_{s} = (1- \zeta(\xi))F_{\ell}(\xi)
\]
and consider the symbols
\[
G_{\pm} = \zeta(\xi) \frac{\a+i\b_f \cdot \xi \pm \sqrt{(\a+i\b_f \cdot \xi)^2 - (1 + (\g_f)^2)(\a^2 - L_{S^n}(\th,\xi))}}{1+|\g_f|^2} + G_s(\xi)
\]
where $L_{S^n}(\th,\xi)$ represents the symbol of the differential operator $L_{S^n}$.  The square root represents the branch of the square root with nonnegative real part.  The argument of the square root lies on the nonpositive real axis only when $\b_f \cdot \xi = 0$ and
\[
L_{S^n} \leq \frac{\a^2 |\g_f|^2}{1 + |\g_f|^2}.
\]
For $\mu$ small enough, this cannot happen on the support of $\zeta$.  Therefore $G_{\pm}$ really are smooth, and hence they really are symbols of order 1 on $\Rn$.  

Now if $T_{a}$ is the operator associated to the symbol $a$,
\begin{eqnarray*}
& & (h\partial_r - \frac{1}{r}T_{G_{+}})(1 + |\g_f|^2)(h\partial_r - \frac{1}{r}T_{G_{-}}) \\
&=& (1 + |\g_f|^2) h^2\partial_r^2 - \frac{1}{r}(1 + |\g_f|^2) (T_{G_{+}} + T_{G_{-}})h\partial_r + \frac{1}{r^2}(1 + |\g_f|^2)T_{G_{+}G_{-}} + hE_1 \\
&=& (1 + |\g_f|^2) h^2\partial_r^2 - \frac{2}{r}(\a + \b_f \cdot h\grad_{\th})h\partial_r T_{\z} + \frac{1}{r^2}(\a^2 + h^2 L_{S^n})T_{\z^2} \\
& & - \frac{2}{r}(1 + |\g_f|^2)T_{G_s} + \frac{1}{r^2}(1 + |\g_f|^2)(T_{G_{+}G_s} + T_{G_{s}G_{-}}+T_{G_s^2}) + hE_1 \\
&=& (1 + |\g_f|^2) h^2\partial_r^2 - \frac{2}{r}(\a + \b_f \cdot h\grad_{\th})h\partial_r T_{\z} + \frac{1}{r^2}(\a^2 + h^2 L_{S^n})T_{\z^2} \\
& & - \frac{2}{r}(1 + |\g_f|^2)T_{G_s} + \frac{1}{r^2}(1 + |\g_f|^2)(T_{G_{+}}T_{G_s} + T_{G_{-}}T_{G_{s}}+T_{G_s}T_{G_s}) + hE_1 \\
\end{eqnarray*}
where $E_1$ is a operator (which changes from line to line as necessary) built of first order semiclassical pseudodifferential operators in $\Rn$ and $\partial_r$ derivatives which is bounded from $H^1(\Rniip)$ to $L^2(\Rniip)$. 

Now let $v = J^{-1}_{\ell} w_{\ell}$. Then
\begin{eqnarray*}
& & (h\partial_r - \frac{1}{r}T_{G_{+}})(1 + |\g_f|^2)(h\partial_r - \frac{1}{r}T_{G_{-}}) v\\
&=& (1 + |\g_f|^2) h^2\partial_r^2 v - \frac{2}{r}(\a + \b_f \cdot h\grad_{\th})h\partial_r T_{\z} v + \frac{1}{r^2}(\a^2 + h^2 L_{S^n})T_{\z^2} v \\
& & - \frac{2}{r}(1 + |\g_f|^2)T_{G_s}v + \frac{1}{r^2}(1 + |\g_f|^2)(T_{G_{+}} + T_{G_{-}}+T_{G_s})T_{G_{s}}v + hE_1v. \\
\end{eqnarray*}
Note that $\hat{w}_{\ell}(r,\xi)$ is only supported for $\xi$ in the support of $(1-\rho)$, and therefore $v = J^{-1}_{\ell} w_{\ell}$ is supported only for $\xi$ in the support of $(1-\rho)$.  Therefore   
\[
T_{\z}v = v. 
\]
since $\z \equiv 1$ on the support of $1 - \rho$.  Similarly, $T_{\z^2}v = v$.  In addition, 
\[
T_{G_s}v = 0,
\]
since $G_s$ is $0$ on the support of $1 - \rho$.  Therefore
\begin{eqnarray*}
& & (h\partial_r - \frac{1}{r}T_{G_{+}})(1 + |\g_f|^2)(h\partial_r - \frac{1}{r}T_{G_{-}}) v \\
&=& (1 + |\g_f|^2) h^2\partial_r^2 v - \frac{2}{r}(\a + \b_f \cdot h\grad_{\th})h\partial_r v + \frac{1}{r^2}(\a^2 + h^2 L_{S^n}) v +hE_1 v \\
&=& \Lphets v +hE_1 v\\
\end{eqnarray*}
where $E_1$ is bounded from $H^1(\Rniip)$ to $L^2(\Rniip)$.

Therefore
\[
\Lphets v = (h\partial_r-\frac{1}{r}T_{G_{+}})z + hE_1 v
\]
for some function $z$, given by 
\[
z = (1 + |\g_f|^2)(h\partial_r - \frac{1}{r}T_{G_{-}}) v. 
\]  
Then 
\begin{eqnarray*}
\hat{g}(r, \xi) &=& \frac{2\mathrm{Re}F_{\ell}-h}{h}\int_1^{\infty} \widehat{\Lphets v}(t,\xi) r^{-\frac{F_{\ell}}{h}} t^{-\frac{\overline{F_{\ell}}}{h}}dt \\
                &=& \frac{2\mathrm{Re}F_{\ell}-h}{h}\int_1^{\infty} \mathcal{F}\left( \left( h\partial_t-\frac{1}{t}T_{G_{+}} \right)z \right) (t,\xi) r^{-\frac{F_{\ell}}{h}} t^{-\frac{\overline{F_{\ell}}}{h}}dt \\
                & & + \frac{2\mathrm{Re}F_{\ell}-h}{h}\int_1^{\infty} h\widehat{E_1 v}(t,\xi) r^{-\frac{F_{\ell}}{h}} t^{-\frac{\overline{F_{\ell}}}{h}}dt \\
\end{eqnarray*}

Integrating by parts gives 
\begin{eqnarray*}
\hat{g}(r, \xi) &=&\frac{2\mathrm{Re}F_{\ell}-h}{h}\int_1^{\infty} \frac{1}{t} \mathcal{F}((T_{\overline{F_{\ell}}} - T_{G_+})z)(t,\xi) r^{-\frac{F_{\ell}}{h}} t^{-\frac{\overline{F_{\ell}}}{h}}dt \\
                & & + \frac{2\mathrm{Re}F_{\ell}-h}{h}\int_1^{\infty} h\widehat{E_1 v}(t,\xi) r^{-\frac{F_{\ell}}{h}} t^{-\frac{\overline{F_{\ell}}}{h}}dt \\
\end{eqnarray*}
There are no boundary terms because $z$ is supported away from $r = 1$.  Therefore by the reasoning used to prove Lemma \ref{largeJHneg}, 
\begin{eqnarray*}
\|g\|^2_{L^2(\Rniip)} &\leq& \|\frac{1}{r} (T_{\overline{F_{\ell}} - G_+})z\|^2_{L^2(\Rniip)} \\
                     & & + h^2\|E_1 v\|^2_{L^2(\Rniip)} \\
\end{eqnarray*}

We need an estimate for $\|\frac{1}{r} (T_{\overline{F_{\ell}} - G_+})z\|^2_{L^2(\Rniip)}$.  Examine the symbol $\overline{F_{\ell}} - G_{+}$.  
\begin{eqnarray*}
& & \overline{F_{\ell}} - G_{+} \\
&=& \overline{F_{\ell}(\xi)} \\
& & - \zeta \frac{\a+i\b_f \cdot \xi + \sqrt{(\a+i\b_f \cdot \xi)^2 - (1 + |\g_f|^2)(\a^2 + L_{S^n}(\th,\xi))}}{1+|\g_f|^2} \\
& & - (1- \zeta)\overline{F_{\ell}(\xi)} \\
&=& \z \left( \overline{F_{\ell}(\xi)} - \frac{\a+i\b_f \cdot \xi + \sqrt{(\a+i\b_f \cdot \xi)^2 -(1+|\g_f|^2)(\a^2+ L_{S^n}(\th,\xi))}}{1+|\g_f|^2} \right) \\
\end{eqnarray*}
On the support of $\zeta$, 
\[
\overline{F_{\ell}(\xi)} = \frac{1}{1+K^2}(1+iK\xi_n + \sqrt{2iK\xi_n-(K\xi_n)^2 +(1+K^2)|\xi|^2 -|K|^2}).  
\]
Therefore
\begin{eqnarray*}
& & \overline{F_{\ell}} - G_{+} \\
&=& \zeta \left( \frac{1+iK\xi_n + \sqrt{2iK\xi_n-(K\xi_n)^2 -(1+K^2)|\xi|^2 -|K|^2}}{1 + K^2} \right. \\
& & - \left. \frac{\a+i\b_f \cdot \xi + \sqrt{(\a+i\b_f \cdot \xi)^2 -(1+(\g_f)^2)(\a^2+ L_{S^n}(\th,\xi))}}{1+|\g_f|^2} \right) \\
&=& \zeta \left( \frac{1 +iK\xi_n}{1+K^2} - \frac{\a+i\b_f \cdot \xi}{1+|\g_f|^2}\right) \\
& & + \zeta \left( \frac{\sqrt{2iK\xi_n-(K\xi_n)^2-(1+K^2)|\xi|^2-|K|^2}}{1+K^2} \right. \\
& & \left. -\frac{\sqrt{(\a+i\b_f \cdot \xi)^2 -(1+(\g_f)^2)(\a^2+ L_{S^n}(\th,\xi))}}{1+|\g_f|^2}\right) \\
\end{eqnarray*}

Consider the first term.  
\begin{eqnarray*}
\frac{1+iK\xi_n}{1+K^2}-\frac{\a+i\b_f \cdot \xi}{1+|\g_f|^2} &=& \frac{(1+|\g_f|^2)(1+iK\xi_n)-(1+K^2)(\a+i\b_f \cdot \xi)}{(1+K^2)(1+|\g_f|^2)} \\
                                                              &=& \frac{(|\g_f|^2-K^2)(1+iK\xi_n)}{(1+K^2)(1+|\g_f|^2)} \\
                                                              & & + \frac{((1+K^2)((1-\a)+i(\b_f- Ke_n)\cdot \xi)}{(1+K^2)(1+|\g_f|^2)} \\
\end{eqnarray*}
The first order operators with symbols 
\[
\frac{(|\g_f|^2-K^2)(1+iK\xi_n)}{(1+K^2)(1+|\g_f|^2)}
\]
and
\[
\frac{((1+K^2)((1-\a)+i(\b_f- Ke_n)\cdot \xi)}{(1+K^2)(1+|\g_f|^2)}
\]
have bounds $\lesssim C_{\mu}$, because they involve multiplication by a function of $\th$ which is bounded by $C_{K}C_{\mu}$.  

Similarly, consider the first order operator with symbol
\begin{eqnarray*}
& & \zeta \left( \frac{\sqrt{2iK\xi_n-(K\xi_n)^2-(1+K^2)|\xi|^2-|K|^2}}{1+K^2} \right. \\
& & \left. -\frac{\sqrt{(\a+i\b_f \cdot \xi)^2 -(1+(\g_f)^2)(\a^2+ L_{S^n}(\th,\xi))}}{1+|\g_f|^2}\right) \\
\end{eqnarray*}
To fit everything horizontally on the page, denote 
\[
\tau_K := 2iK\xi_n-(K\xi_n)^2-(1+K^2)|\xi|^2-|K|^2
\]
and
\[
\tau_f := (\a+i\b_f \cdot \xi)^2 -(1+(\g_f)^2)(\a^2+ L_{S^n}(\th,\xi)).
\]
Then
\begin{eqnarray*}
& & \frac{\sqrt{\tau_K}}{1+K^2}-\frac{\sqrt{\tau_f}}{1+|\g_f|^2}\\
&=& \frac{(1+|\g_f|^2)\sqrt{\tau_K} - (1+K^2)\sqrt{\tau_f}}{(1+K^2)(1+|\g_f|^2)} \\
&=& \frac{(1+|\g_f|^2)^2 \tau_K - (1+K^2)^2 \tau_f}{(1+K^2)(1+|\g_f|^2)((1+|\g_f|^2)\sqrt{\tau_K} + (1+K^2)\sqrt{\tau_f})} \\
&=& (1+K^2)\frac{\tau_K-\tau_f}{(1+|\g_f|^2)((1+|\g_f|^2)\sqrt{\tau_K} + (1+K^2)\sqrt{\tau_f})} \\
& & + \frac{((1+|\g_f|^2)^2 - (1 +K^2)^2) \tau_K}{(1+K^2)(1+|\g_f|^2)((1+|\g_f|^2)\sqrt{\tau_K} + (1+K^2)\sqrt{\tau_f})}. \\
\end{eqnarray*}

Expanding,
\begin{eqnarray*}
& & \tau_K-\tau_f \\
&=& 2i(Ke_n-\a\b_f)\cdot \xi+ ((\b_f \cdot \xi)^2 -(Ke_n \cdot \xi)^2)+(|\g_f|^2-K^2)L(\th,i\xi) \\
& & +(|\g_f|^2-|K|^2) +(1 + K^2)(|\xi|^2 - L(\th,\xi)).
\end{eqnarray*}
Therefore the second term has operator bounds $\lesssim C_{\mu}$, because each term involves multiplication by a function of $\th$ which is bounded by $C_{K}C_{\mu}$.  

Therefore
\[
\|\frac{1}{r} (T_{\overline{F_{\ell}} - G_+})z\|^2_{L^2(\Rniip)} \leq \d^2 \|z\|^2_{H^1(\Rniip)}
\]
for $\mu$ small enough.  Then
\begin{eqnarray*}
\|g\|^2_{L^2(\Rniip)} &\lesssim& \|\frac{1}{r} (T_{\overline{F_{\ell}} - G_+})z\|^2_{L^2(\Rniip)} \\
                      & & + h^2\|E_1 v\|^2_{L^2(\Rniip)} \\
                      &\lesssim& \d^2 \|z\|^2_{H^1(\Rniip)} + h^2\|E_1 v\|^2_{L^2(\Rniip)}\\
                      &\lesssim& \d^2 \|z\|^2_{H^1(\Rniip)} + h^2\| v\|^2_{H^1(\Rniip)}\\
\end{eqnarray*}

Since
\[
\Lphets v = (h\partial_r-\frac{1}{r}T_{G_{+}})z + hE_1 v,
\]
we have
\begin{eqnarray*}
\| \Lphets v \|^2_{L^2(\Rniip)} &\geq& \|(h\partial_r-\frac{1}{r}T_{G_{+}})z \|^2_{L^2(\Rniip)} - h^2\|E_1 v\|^2_{L^2(\Rniip)} \\
                                &\geq& \|J^{*}_{\ell} z\|^2_{L^2(\Rniip)} -\|\frac{1}{r}T_{\overline{F_{\ell}} - G_+} z\|^2_{L^2(\Rniip)} -h^2\| v\|^2_{H^1(\Rniip)} \\
                                &\gtrsim& \|z\|^2_{H^1(\Rniip)} - \d^2 \| z\|^2_{H^1(\Rniip)} - h^2\|v\|^2_{H^1(\Rniip)} \\
                                &\gtrsim& \|z\|^2_{H^1(\Rniip)} - h^2 \|v\|^2_{H^1(\Rniip)} \\
\end{eqnarray*}
for $\d$ small enough.  Therefore
\begin{eqnarray*}
\|g\|^2_{L^2(\Rniip)} &\lesssim& \d^2 \| \Lphets v \|^2_{H^1(\Rniip)} + h^2\| v\|^2_{H^1(\Rniip)} \\
                      &\lesssim& \d^2 \| \Lphets v \|^2_{H^1(\Rniip)} + h^2\| J^{-1}_s (1 - P) w \|^2_{H^1(\Rniip)} \\
\end{eqnarray*}

Using similar reasoning as for the small frequency case, 
\[
h^2\| J^{-1}_s (1 - P) w \|^2_{H^1(\Rniip)} \lesssim h^2\| J^{-1}_s (1 - P) w \|^2_{H^1_r(\Rniip)}.
\]
Therefore
\begin{eqnarray*}
\|g\|^2_{L^2(\Rniip)} &\lesssim& \d^2 \| \Lphets v \|^2_{H^1(\Rniip)} + h^2\| J^{-1}_s (1 - P) w \|^2_{H^1_r(\Rniip)} \\
                      &\lesssim& \d^2 \| \Lphets v \|^2_{H^1(\Rniip)} + h^2\| w_{\ell} \|^2_{L^2(\Rniip)} \\
\end{eqnarray*}

Then for $\d$ small enough,
\[
\|g\|_{L^2(\Rniip)} \lesssim \half \| \Lphets v \|_{L^2(\Rniip)}+ h \|w_{\ell} \|_{L^2(\Rniip)}.
\]
Now using \eqref{largefirstpart} and Lemma \ref{largeJHneg},  
\[
\frac{h}{\sqrt{\e}} \|w_{\ell} \|_{L^2(\Rniip)} \lesssim \|J_{\ell}\Lphets \chi_2 J^{-1}_{\ell} w_{\ell} \|_{H^{-1}_r(\Rniip)} + h \|w_{\ell} \|_{L^2(\Rniip)}+ O(h^{\infty})\|w \|_{L^2(\Rniip)}.
\]
Absorbing the second last term into the left side gives
\[
\frac{h}{\sqrt{\e}} \|w_{\ell} \|_{L^2(\Rniip)} \lesssim \|J_{\ell}\Lphets \chi_2 J^{-1}_{\ell} w_{\ell} \|_{H^{-1}_r(\Rniip)} + O(h^{\infty})\|w \|_{L^2(\Rniip)}.
\]
We can finish the argument as in the small frequency case to get
\[
\frac{h}{\sqrt{\e}}\|w_{\ell}\|_{L^2(\Rniip)} \lesssim \|\Lphets w_{\ell} \|_{H^{-1}(\Rniip)}+ O(h)\|w \|_{L^2(\Rniip)}.
\]
This finishes the proof of Lemma \ref{largeCarl}, and thus of Proposition \ref{simpleCarl}.

\section{Proof of Theorem \ref{mainCarl}}

%arxivfinal

Now I can prove Proposition \ref{specCarl} essentially by gluing together estimates of the form in Proposition \ref{simpleCarl}.  First note that by a change of variables similar to the ones in Section 2, we can show that if for all $\theta \in S^n$ such that some $(r,\th)$ is in $\Om_2$, 
\[
|\sin(\th_k)- 1| \leq \mu \mbox{ where } k = 1, \ldots, n-1
\]
and
\[
|\grad_{S^n}\log f - K e_n |_{S^n} \leq \mu,
\]
where $\mu$ is small enough, then
\begin{equation}\label{goodCarl}
\frac{h}{\sqrt{\e}}\| w\|_{L^2(\Om)} \lesssim \|\Lphe w \|_{H^{-1}(A_O)}
\end{equation}
for all $w \in \Czinf(\Om)$.

Now let $\Om$ be as in Proposition \ref{specCarl}.  We can take an open cover $U_1, \ldots, U_m$ of $\Om$ such that on each $\Om \cap U_j$, there exists $K_j$ such that under some choice of coordinates, $|\grad_{S^n} \log f - K_je_n| \leq \mu_{K_j}$ and $|\sin(\th_k)- 1| \leq \mu_{K_j}$, where $\mu_{K_j}$ is the value of $\mu$ from Proposition 1 which works for $K = K_j$.  (Since $|\grad_{S^n} \log f|$ must be bounded above, $\mu_{K_j}$ must be bounded below, and therefore this is possible with only finitely many $U_j$.)  

Let $\z_1, \ldots \z_m$ be a smooth partition of unity subordinate to the cover $U_1, \ldots U_m$.  Now for $w \in \Czinf(\Om)$, 
\[
w = \z_1 w + \ldots \z_m w =: w_1 + \ldots + w_m,
\]
where each $w_j \in \Czinf(\Om \cap U_j)$.  
Applying the result \eqref{goodCarl} to the domain $\Om \cap U_j$,
\[
\frac{h}{\sqrt{\e}}\| w_j \|_{L^2(\Om \cap U_j)} \lesssim \|\Lphe w_j \|_{H^{-1}(A_O)}
\]
for each $j = 1, \ldots, m$.  Then
\[
\sum_j \frac{h}{\sqrt{\e}}\| w_j \|_{L^2(\Om)} \lesssim \sum_j \|\Lphe w_j \|_{H^{-1}(A_O)},
\]
so
\[
\frac{h}{\sqrt{\e}}\| w \|_{L^2(\Om)} \lesssim \sum_j \|\Lphe w_j \|_{H^{-1}(A_O)}.
\]

Now by the product rule,
\begin{eqnarray*}
\|\Lphe w_j \|_{H^{-1}(A_O)} &=& \| \Lphe \z_j w \|_{H^{-1}(A_O)} \\
                              &\leq& \| \z_j \Lphe w \|_{H^{-1}(A_O)} + Ch\|w\|_{A_O} \\
                              &\leq& \| \Lphe w \|_{H^{-1}(A_O)} + Ch\|w\|_{A_O}. \\
\end{eqnarray*}
Therefore
\begin{equation}\label{zeroCarl}
\frac{h}{\sqrt{\e}}\| w \|_{L^2(\Om)} \lesssim \|\Lphe w \|_{H^{-1}(A_O)}.
\end{equation}
for $\e$ small enough, for every $w \in \Czinf(\Om)$.

To treat the case where $W$ and $q$ are non-zero, note that
\[
\Lphaqe = \Lphe + h(W \cdot hD + hD \cdot W) + 2ihW\cdot \grad (\log r + h\frac{\log^2 r}{2\e}) + h^2(q + W^2).
\]
Therefore 
\[
\frac{h}{\sqrt{\e}}\| w \|_{L^2(\Om)} \lesssim \|\Lphaqe w \|_{H^{-1}(A_O)} + hC\|w\|_{L^2(A_O)}
\]
and the last term can be absorbed into the left side to give
\[
\frac{h}{\sqrt{\e}}\| w \|_{L^2(\Om)} \lesssim \|\Lphaqe w \|_{H^{-1}(A_O)}.
\]

This completes the proof of Proposition \ref{specCarl}.

Finally, I can prove Theorem \ref{mainCarl} by gluing together estimates of the form in Proposition \ref{specCarl}.  If $E$ is a compact subset of $F_{\Om} \setminus Z_{\Om}$, then define $\Om '$ to be a smooth domain containing $\Om$, with $\partial \Om \cap \partial \Om ' = E$.   

Then let $U_1, \ldots, U_m$ be an open cover of $\Om$ such that each $\partial U_j \cap E$ coincides with a graph of the form $r = f_j(\th)$.  For each $U_j$, Proposition \ref{specCarl} gives us
\[
\frac{h}{\sqrt{\e}}\| w \|_{L^2(U_j)} \lesssim \|\Lphaqe w \|_{H^{-1}(A_j)}
\]
for $w \in \Czinf(U_j)$.  

Each $A_j$ is defined by the graph of a function $r = f_j(\th)$, and since $\partial \Om '$ is smooth and coincides with $\partial \Om$ on $E$, and $E$ is a compact subset of $F_{\Om} \setminus Z_{\Om}$, $\partial \Om '$ must be locally a graph in a neighbourhood of $E$.  Therefore we can assume that $A_j$ coincides with $\Om '$ in a neighbourhood of each $U_j$, in the sense that their characteristic functions are equal in that neighbourhood.  Then there is a smooth cutoff function $\chi_j$ defined on $A_j \cap \Om'$ which is identically one on $U_j$ but vanishes outside on the complements of $A_j$ and $\Om '$.  Multiplication by this function provides a bounded map from $H^1_0(A_j)$ to $H^1_0(\Om ')$ and vice versa, and therefore for $w \in \Czinf(U_j)$,  $\|w\|_{H^{-1}(\Om ')} \simeq \|w\|_{H^{-1}(A_j)}$.  Therefore we have
\[
\frac{h}{\sqrt{\e}}\| w \|_{L^2(U_j)} \lesssim \|\Lphaqe w \|_{H^{-1}(\Om ')}
\]
for $w \in \Czinf(U_j)$.  

Gluing together these estimates in the matter used above gives
\[
\frac{h}{\sqrt{\e}}\| w \|_{L^2(\Om)} \lesssim \|\Lphaqe w \|_{H^{-1}(\Om ')}
\]
for $w \in \Czinf(\Om)$.  

Finally,  note that if $w \in C^{\infty}_0(\Om)$, then $e^{\frac{(\log r)^2}{\e}}w \in C^{\infty}_0(\Om)$, so
\[
\frac{h}{\sqrt{\e}}\|e^{\frac{(\log r)^2}{\e}} w\|_{L^2(\Om)} \lesssim \| e^{\frac{(\log r)^2}{\e}} \Lphaq w \|_{H^{-1}(\Om ')}.
\]
On $\Om$, there exists some $C_{\Om}$ such that $1 \leq e^{\frac{(\log r)^2}{\e}} \leq e^{\frac{C_{\Om}}{\e}}$, so
\[
h \|w\|_{L^2(\Om)} \lesssim \|\Lphaq w \|_{H^{-1}(\Om ')}  
\]
as desired.  This establishes Theorem \ref{mainCarl}.  
\vspace{4mm}

\textbf{Remark} If we want to prove Theorem \ref{mainthm2} instead of Theorem \ref{mainthm}, then we could begin by supposing that $f : S^n \rightarrow (0,\infty)$ is a $\Cinf$ function such that $\Om$ lies entirely in the region $A_I = \{ (r,\th) | r \leq f(\th) \} \subset \RnI$, and $E$ is a subset of the graph $r = f(\th)$.  Then by the change of variables $(r,\th) \mapsto (\frac{1}{r},\th)$, $\Om$ maps to a region $\hat{\Om}$of the form described in Proposition \ref{specCarl}.  Therefore by \eqref{zeroCarl}, 
\[
h \|w\|_{L^2(\hat{\Om})} \lesssim \|\Lphe w \|_{H^{-1}(\hat{A}_O)}.
\]
for $w \in \Czinf{\hat{\Om}}$, where $\ph = \log r$.  Changing variables back gives the Carleman estimate
\[
h \|w\|_{L^2(\Om)} \lesssim \|\mathcal{L}_{-\log r,\e} w \|_{H^{-1}(A_I)}
\]
for $w \in \Czinf{\Om}$.  Therefore by the same kind of argument as above, we get 
\[
h \|w\|_{L^2(\Om)} \lesssim \|\Lphaq w \|_{H^{-1}(\Om ')}  
\]
where $\ph = -\log r$, and $\Om '$ is a domain containing $\Om$, with $E \subset \partial \Om ' \cap \partial \Om$, whenever $E$ is of the form described in Theorem \ref{mainthm2}. Using this Carleman estimate in the place of Theorem \ref{mainCarl} in the remainder of the argument proves Theorem \ref{mainthm2} instead of Theorem \ref{mainthm}.

\section{Complex Geometric Optics Solutions}

Theorem \ref{mainCarl} can be used to construct solutions to equations of the system \eqref{pDirichletProblem}.  The key is the following proposition.

\begin{prop}\label{HBprop}

For every $v \in L^2(\Om)$, there exists $u \in H^1(\Om)$ such that 
\begin{eqnarray*}
\Lphaq^{*} u &=& v  \mbox{ on } \Om\\
u|_{E}     &=& 0 
\end{eqnarray*}

and
\[
\|u\|_{H^1(\Om)} \lesssim \frac{1}{h}\|v\|_{L^2(\Om)}.
\] 

\end{prop}

\begin{proof}
The proof is based on a Hahn-Banach argument.  Suppose $v \in L^2(\Om)$.  Then for all $w \in C^{\infty}_0(\Om)$,
\[
|(w|v)_{\Om}| \lesssim \frac{1}{h}\|v\|_{L^2(\Om)} h \|w\|_{L^2(\Om)}.
\]
Therefore, by Theorem \ref{mainCarl},
\begin{equation}\label{useCarl}
|(w|v)_{\Om}| \lesssim \frac{1}{h}\|v\|_{L^2(\Om)} \|\Lphaq w \|_{H^{-1}(A_O)}.
\end{equation}
Now consider the subspace
\[
\{ \Lphaq w | w \in C^{\infty}_0(\Om)\} \subset H^{-1}(A_O).
\]
By the estimate from Theorem \ref{mainCarl}, the map $ \Lphaq w \longmapsto (w|v)_{\Om}$ is well-defined on this subspace.  It is a linear functional, and by \eqref{useCarl}, it is bounded by $\frac{C}{h}\| v\|_{L^2(\Om)}$. 

Therefore by Hahn-Banach, there exists an extension of this functional to the whole space $H^{-1}(A_O)$ with the same bound.  This can be represented by an element of the dual space $H^1_0(A_O)$, so there exists $u \in H^1_0(A_O)$ such that 
\[
\|u \|_{H^1(A_O)} \lesssim \frac{1}{h}\|v\|_{L^2(\Om)}
\]
and
\begin{eqnarray*}
(w|v)_{\Om} &=& (\Lphaq w | u)_{A_O} \\
            &=& (\Lphaq w | u)_{\Om} \\
\end{eqnarray*}
for all $w \in C^{\infty}_0(\Om)$.  Note that $u \in H^1_0(A)$ implies that $u|_E = 0$.  Then 
\[
(w|v)_{\Om} = (w |\Lphaq^{*} u)_{\Om}
\]
since $w \in C^{\infty}_0(\Om)$, and thus
\[
(w| v - \Lphaq^{*} u)_{\Om} = 0
\]
for all $v \in C^{\infty}_0(\Om)$.  Therefore $v = \Lphaq^{*} u$ on $\Om$, and 
\[
\|u \|_{H^1(\RnI)} \lesssim \frac{1}{h}\|v\|_{L^2(\Om)}
\]
as desired.  

\end{proof}

Now I can construct the complex geometrical optics solutions.  

\begin{prop}\label{boundarysolns}

There exists a solution of the problem 
\begin{eqnarray*}
\Laq u &=& 0 \mbox{ on } \Om\\
u|_{E} &=& 0 
\end{eqnarray*}
of the form $u = e^{\frac{1}{h}(\ph + i \psi)}(a+r) - e^{\frac{\ell}{h}} b$, where $\ph(x,y) = \log r$, $\psi$ is a solution to the eikonal equation $\grad \ph \cdot \grad \psi = 0,  |\grad \ph| = |\grad \psi|$; $a$ and $b$ are $C^2$ functions on $\Om$; and 
\[
\mathrm{Re} \, \ell(x,y) = \ph(x,y) - k(x,y)
\]
where $k(x) \simeq \mathrm{dist}(x,E)$ in a neighbourhood of $E$, and $b$ has its support in that neighbourhood.  Finally, $r \in H^1(\Om)$, with $r|_{E} = 0$, $\|r\|_{H^1(\Om)} = O(h)$, and $\|r\|_{L^2(\partial \Om)} = O(h^{\half})$.
\end{prop}

The proof is a combination of the proofs of the equivalent theorems in [DKSU] and [KSU].

\begin{proof}
Let $\ph(r,\th) = \log r$, and take $\psi(r,\th) = d_{S^n}(\th, \omega)$ for some fixed point $\omega \in S^n$.  If $\omega \neq \theta$ for all $(r,\th) \in \Om$, then $\psi$ solves the eikonal equation $\grad \ph \cdot \grad \psi = 0,  |\grad \ph| = |\grad \psi|$.  Then
\begin{eqnarray*}
h^2 \Laq e^{\frac{1}{h}(\ph + i\psi)} &=& e^{\frac{1}{h}(\ph + i\psi)}(h(D+W) \cdot (\grad \psi - i\grad \ph) \\
                                      & & + h(\grad \psi - i\grad \ph) \cdot (D+W) + h^2 \Laq).
\end{eqnarray*}
Therefore if $a$ is a $C^2$ solution to 
\[
(\grad \psi - i\grad \ph) \cdot Da + (\grad \psi - i\grad \ph) \cdot Wa + \frac{1}{2i}(\Lap \psi - i \Lap \ph)a = 0,
\]
then 
\[
h^2 \Laq e^{\frac{1}{h}(\ph + i\psi)} a = e^{\frac{1}{h}(\ph + i\psi)} h^2 \Laq a = O(h^2)e^{\frac{1}{h}(\ph + i\psi)}.
\]
We can look for an exponential solution $a = e^{\Phi}$, in which case the relevant equation becomes 
\[
(\grad \ph + i \grad \psi) \cdot \grad \Phi + i(\grad \ph + i \grad \psi) \cdot W  + \half \Lap(\ph + i \psi)= 0.
\]
Now suppose $x \in \RnI$, and write $x = (x_{\omega}, x')$, where $x_{\omega}$ is the component of $x$ in the $\omega$ direction, and $x'$ are the remaining components.  Then by considering $z = x_{\omega} + i |x'|$ as a complex variable, we get $\ph = \mathrm{Re}\log z$ and $\psi = \mathrm{Im} \log z$.  Now our equation is an inhomogeneous Cauchy-Riemann equation in the $z$ variable, and can be solved by the Cauchy formula.  Then $a$ is $C^2$, since $W$ is.  Note that the solution is only unique up to addition of terms $g_{a}$ with 
\begin{equation}\label{hCR1}
(\grad \ph + i \grad \psi) \cdot \grad g_a = 0.
\end{equation}

Now I want to construct a (complex valued) function $\ell$ to be an approximate solution to the equation 
\begin{equation*}
\begin{split}
\grad \ell \cdot \grad \ell &= 0 \\
                    \ell|_E &= \ph + i\psi.
\end{split}
\end{equation*}
In order to avoid duplicating the solution $\ph + i\psi$, we can ask for
\[
\nd \ell|_{E} = -\nd(\ph + i\psi)|_{E}.
\]
To construct an approximate solution, pick coordinates $(t,s)$ near $E$ such that $t$ are the coordinates along $E$ and $s$ is perpendicular to $E$.  Suppose $\ell$ takes the form of a power series
\[
\ell(t,s) = \sum_{j = 0}^{\infty} a_j(t)s^j.
\]
Then
\begin{eqnarray*}
\grad \ell &=& (\grad_t \ell, \partial_s \ell) \\
           &=& (\sum_{j=0}^{\infty}\grad_t a_j(t) s^j, \sum_{j=0}^{\infty}a_j(t)js^{j-1}) 
\end{eqnarray*}
Expanding the equation $\grad \ell \cdot \grad \ell = 0$ gives
\begin{eqnarray*}
0 &=& \sum_{j+k=0}\grad_t a_j \grad_t a_k + \left( \sum_{j+k=1}\grad_t a_j \grad_t a_k \right)s + \left( \sum_{j+k=2}\grad_t a_j \grad_t a_k \right)s^2+\ldots \\
  & & + \sum_{j+k=2} jka_j a_k + \left( \sum_{j+k=3} jka_j a_k \right)s + \left( \sum_{j+k=4} jka_j a_k \right)s^2 + \ldots \\
\end{eqnarray*}
Considering each power of $s$ separately gives a sequence of equations
\begin{equation}\label{formaleqn}
\sum_{j+k=m}\grad_t a_j \grad_t a_k + \sum_{j+k=m+2} jka_j a_k = 0.
\end{equation}
for each $m = 0,1,2,\ldots$.  The boundary conditions determine $a_0$ and $a_1$, so we can solve this recursively.  If $m \geq 1$, and all $a_j$ are known for $ j \leq m$, the only part of \eqref{formaleqn} which contains an unknown looks like $2(m+1)a_1 a_{m+1}$.  Note that 
\[
a_1 = -\nd(\ph + i\psi)
\]
Since $E$ coincides with a graph $r = f(\th)$ for some smooth function $f$, and $\ph = \log r$, there exists some $\e_0 > 0$ so that $|a_1| > \e_0$ on $E$, and so we can divide by $a_1$ to solve for $a_{m+1}$.  

This gives a formal power series for $\ell$, which may or may not converge outside $s = 0$.  However, we can construct a $C^{\infty}$ function in $\Om$ whose Taylor series in $s$ coincides with this formal power series at $s = 0$.  

Let $\chi : \R \rightarrow [0,1]$ be a smooth cutoff function which is identically one on $[-\half, \half]$, and identically zero outside $(-1,1)$.  Set 
\[
\ell = \sum_{j=0}^{\infty} a_j(t)\chi(b_j s)s^j = \sum_{j=0}^{\infty} c_j(t,s)s^j
\]
where 
\[
b_j = \max_{k \leq j} \{ \| a_k(t)\|_{C^k(\Om)}, 1 \}.
\]
This defines a $C^{\infty}$ function on $\Om$ whose Taylor series at $s = 0$ coincides with the formal power series calculated earlier.  Now examine the expression $\grad \ell \cdot \grad \ell$.  The coefficient of $s^m$ in this expansion is
\[
\sum_{j+k=m}\grad_t c_j \grad_t c_k + \sum_{j+k=m}(\partial_s c_j + (j+1)c_{j+1})(\partial_s c_k + (k+1)c_{k+1}).
\]
For $s \leq \frac{1}{2b_p}$, all $c_j(t,s) = a_j(t)$ for $ j \leq p$, so this is exactly zero by design when $s \leq \frac{1}{2b_p}$ and $m < p$.  Therefore on the region where $s \leq \frac{1}{2b_p}$, $\grad \ell \cdot \grad \ell$ is a smooth function which is $O(s^p)$ as $s$ goes to zero.  Since this can be done for any $p$, 
\[
\grad \ell \cdot \grad \ell = O(\mathrm{dist}(x, E)^{\infty}).
\]
Moreover, 
\[
\nd \mathrm{Re} \, \ell|_{E} = -\nd \ph|_{E} < -\e_0,
\]
and
\[
\mathrm{Re} \, \ell|_{E} = \ph|_{E}
\]
so in a neighbourhood of of $E$, 
\begin{equation}\label{ellgood}
\mathrm{Re} \, \ell(x,y) = \ph(x,y) - k(x,y)
\end{equation}
where $k(x) \simeq \mathrm{dist}(x,E)$ in a neighbourhood of $E$.

By a similar method, we can construct an approximate solution $b$ for the problem 
\begin{eqnarray*}
\grad \ell \cdot D b + \grad \ell \cdot W b &=& 0 \\
                                     b|_{E} &=& a|_{E}.
\end{eqnarray*}
so 
\begin{eqnarray*}
\grad \ell \cdot D b + \grad \ell \cdot W b &=& O(\mathrm{dist}(x,E)^{\infty}) \\
                                     b|_{E} &=& a|_{E}.
\end{eqnarray*}
Multiplying $b$ by a smooth cutoff function does not change these properties, so we may as well assume that $b$ is only supported close to $E$ for \eqref{ellgood} to hold.  Then 
\[
-h^2 \Laq(e^{\frac{\ell}{h}}b) = e^{\frac{\ell}{h}} (O(\mathrm{dist}(x,E)^{\infty}) + O(h^2)),
\]
so
\[
|h^2 \Laq(e^{\frac{\ell}{h}}b)| = e^{\frac{\ph}{h}}e^{\frac{-k}{h}}(O(\mathrm{dist}(x,E)^{\infty}) + O(h^2)).
\]
If $\mathrm{dist}(x,E) \leq h^{\half}$, for $h$ small, this is $e^{\frac{\ph}{h}}O(h^2)$, because of the $O(\mathrm{dist}(x,E)^{\infty})$ term.  On the other hand, if $\mathrm{dist}(x,E) \geq h^{\half}$, this is still $e^{\frac{\ph}{h}}O(h^2)$, because of $e^{\frac{-k}{h}}$.  

Now $e^{\frac{1}{h}(\ph + i \psi)}a - e^{\frac{\ell}{h}} b = 0$ on $E$, and 
\[
e^{-\frac{\ph}{h}} h^2 \Laq (e^{\frac{1}{h}(\ph + i\psi)} a + e^{\frac{\ell}{h}}b) = v
\]
where $\|v\|_{L^2(\Om)} = O(h^2)$.  By Proposition \ref{HBprop}, the problem 
\begin{eqnarray*}
\Lphaq^{*} r_1 =  e^{-\frac{\ph}{h}} h^2 \Laq e^{\frac{\ph}{h}} r_1 &=& -v  \mbox{ on } \Om\\
                                                           r_1|_{E} &=& 0 
\end{eqnarray*}
has an $H^1$ solution $r_1$ with 
\[
\|r_1\|_{H^1(\Om)} \lesssim \frac{1}{h}\|v\|_{L^2(\Om)} = O(h)
\]
Set $r = e^{-\frac{i\psi}{h}}r_1$, and $u = e^{\frac{1}{h}(\ph + i \psi)}(a+r) - e^{\frac{\ell}{h}} b$.  Then
\[
\|r\|_{H^1(\Om)} = O(h),
\]
so $\|r\|_{L^2(\partial \Om)} = O(h^{\half})$ by the trace theorem, and
\begin{eqnarray*}
\Laq u &=& 0 \mbox{ on } \Om\\
u|_{E} &=& 0. 
\end{eqnarray*}
This finishes the proof.
\end{proof}

If the boundary condition is not needed, then the result is as follows:

\begin{prop}\label{noboundarysolns}
There exists a solution of the problem 
\[
\Laq u = 0 \mbox{ on } \Om
\]
of the form $u = e^{\frac{1}{h}(\ph + i \psi)}(a+r)$, where $\ph(x,y)$ is any limiting Carleman weight, $\psi$ is any solution to the eikonal equation, $a$ is a $C^2$ function on $\Om$, and $r \in H^1(\Om)$, with $\|r\|_{H^1(\Om)} = O(h)$, and $\|r\|_{L^2(\partial \Om)} = O(h^{\half})$.
\end{prop}

This is essentially lemma 3.4 from [DKSU].  Note that we can always replace $a$ by $\g a$, where $\g$ is a solution to 
\[
(\grad \ph + i \grad \psi) \cdot \grad \g = 0.
\]
on $\Om$.

\section{Proof of Theorem \ref{mainthm}}

For convenience, $\| \cdot \|$ will denote the $L^2$ norm in this section, unless otherwise indicated.  The tilde as used in this section has nothing to do with the notation from section 2.  

Using Proposition \ref{boundarysolns}, we can construct $\tilde{u}_2 = e^{\frac{1}{h}(\ph + i \psi)}(a_2+r_2) - e^{\frac{\ell}{h}} b =: u_2 + u_r$ to be a solution to 
\begin{eqnarray*}
\Laqtwo \tilde{u}_2 &=& 0 \mbox{ on } \Om\\
\tilde{u}_2|_{E} &=& 0. 
\end{eqnarray*}
Then $-\ph$ is also a Carleman weight, and if $\ph$ and $\psi$ satisfy the eikonal equation, then so do $-\ph$ and $\psi$.  Therefore using Proposition \ref{noboundarysolns}, we can construct $u_1 = e^{\frac{1}{h}(-\ph + i\psi)}(a_1 + r_1)$ to be a solution to 
\[
\L_{W_1,\overline{q_1}}u_1 = 0.
\]
Let $w$ be the unique solution to 
\begin{eqnarray*}
\Laqone w         &=& 0 \\
w|_{\partial \Om} &=& \tilde{u}_2 |_{\partial \Om}.
\end{eqnarray*}
(This is where we use the assumption that $\Laqone$ does not have a zero eigenvalue.)  Note that in particular, $w|_E = \tilde{u}_2 |_E = 0$, so by the hypothesis on the Dirichlet-Neumann map, 
\[
\nd (w - \tilde{u}_2)|_{U} = 0. 
\]

Now 
\begin{equation}\label{nonGF}
\begin{split}
\L_{W_1, q_1}(w-\tilde{u}_2) &= -\L_{W_1, q_1}\tilde{u}_2 \\
                             &= (\L_{W_2,q_2} - \L_{W_1, q_1})\tilde{u}_2 \\
                             &= (W_2 - W_1) \cdot D \tilde{u}_2 + D \cdot (W_2 - W_1) \tilde{u}_2 + (W_2^2 - W_1^2 + q_2 - q_1) \tilde{u}_2.
\end{split}
\end{equation}
On the other hand, the Green's formula from [DKSU] gives us that
\begin{equation}\label{GF}
\begin{split}
\int_{\Om} \L_{W_1, q_1}(w-\tilde{u}_2) \overline{u_1} dV &= \int_{\partial \Om} \nd (\tilde{u}_2 - w) \overline{u_1} dS \\
                                                          &= \int_{\partial \Om \setminus U} \nd (\tilde{u}_2-w) \overline{u_1} dS.
\end{split}                                                          
\end{equation}
Combining \eqref{nonGF} with \eqref{GF} gives
\begin{eqnarray*}
\int_{\partial \Om \setminus U} \nd (\tilde{u}_2-w) \overline{u_1} dS &=& \int_{\Om}(W_2 - W_1) \cdot(D\tilde{u}_2 \overline{u_1} + \tilde{u}_2 \overline{Du_1})dV \\
                                                                             & & + \int_{\Om} (W_2^2 - W_1^2 + q_2 - q_1) \tilde{u}_2 \overline{u_1} dV.
\end{eqnarray*}
Expanding $\tilde{u}_2$ as $\tilde{u}_2 = u_2 + u_r$ on the right side gives
\begin{equation}\label{bGf}
\begin{split}
\int_{\partial \Om \setminus U}\nd (\tilde{u}_2-w)\overline{u_1}dS &= \int_{\Om}(W_2 - W_1) \cdot(D u_2 \overline{u_1} + u_2 \overline{Du_1})dV \\
                                                                          &  + \int_{\Om} (W_2^2 - W_1^2 + q_2 - q_1) u_2 \overline{u_1} dV \\
                                                                          &  + \int_{\Om}(W_2 - W_1) \cdot(D u_r \overline{u_1} + u_r \overline{Du_1})dV \\
                                                                          &  + \int_{\Om} (W_2^2 - W_1^2 + q_2 - q_1) u_r \overline{u_1} dV \\
\end{split}                                                                          
\end{equation}

We can label the terms as follows:   \textcircled{1} = \textcircled{2} + \textcircled{3} + \textcircled{4} + \textcircled{5}.  Consider the terms on the right side first.  \textcircled{2} is bounded above by 
\[
\|(W_2 - W_1) e^{-\frac{\ph}{h}} D u_2 \|_{\Om} \|a_1 + r_1\|_{\Om} + \|(W_2 - W_1) e^{\frac{\ph}{h}} \overline{Du_1} \|_{\Om} \|a_2 + r_2\|_{\Om}.
\]
Since $W_2 - W_1$ is bounded on $\Om$, $\|a_1\|_{\Om}$ and $\|a_2\|_{\Om}$ are $O(1)$, and $\|r_1\|_{\Om}$ and $\|r_2\|_{\Om}$ are $O(h)$, 
\[
| \mbox{\textcircled{2}} | \lesssim \| e^{-\frac{\ph}{h}} D u_2 \|_{\Om} + \| e^{\frac{\ph}{h}} D u_1 \|_{\Om}.
\] 
Meanwhile, \textcircled{3} is bounded above by 
\[
| \mbox{\textcircled{3}} | \leq \| (W_2^2 - W_1^2 + q_2 - q_1) (a_2+r_2) \|_{\Om} \|a_1 + r_1\|_{\Om} = O(1).
\]
Similarly, 
\begin{eqnarray*}
| \mbox{\textcircled{4}} | &\lesssim& \|e^{-\frac{\ph}{h}}Du_r\|_{\Om}+ \|e^{\frac{\ph}{h}}\overline{Du_1}\|_{\Om} \|e^{\frac{-2\b y}{h}}\|_{\Om} \\
                           &\lesssim& \|e^{-\frac{\ph}{h}}Du_r\|_{\Om}+ h \|e^{\frac{\ph}{h}}\overline{Du_1}\|_{\Om} 
\end{eqnarray*}
and \textcircled{5} is bounded above by
\[
| \mbox{\textcircled{5}} | \leq \| (W_2^2 - W_1^2 + q_2 - q_1) e^{\frac{-2\b y}{h}} b \|_{\Om} \|a_1 + r_1\|_{\Om} = O(h).
\]

Now examine term \textcircled{1}:
\[
\left | \int_{\partial \Om \setminus U}\nd(\tilde{u}_2-w)\overline{u_1}dS \right | \leq \| \nd(\tilde{u}_2-w)e^{-\frac{\ph}{h}}\|_{\partial \Om \setminus U}
                                                                                               \|a_1+r_1\|_{\partial \Om \setminus U}. 
\]
The factor $\|a_1+r_1\|_{\partial \Om \setminus U}$ is $O(1)$.  Furthermore, on $\partial \Om \setminus U$, $\nd \ph \geq \e_1$, so 
\begin{eqnarray*}
\left | \int_{\partial \Om \setminus U}\nd(\tilde{u}_2-w)\overline{u_1}dS \right | 
                                            &\lesssim& \frac{1}{\sqrt{\e_1}}\| \sqrt{\nd \ph} e^{-\frac{\ph}{h}} \nd (\tilde{u}_2-w)\|_{\partial \Om \setminus U} \\
                                            &\lesssim& \frac{1}{\sqrt{\e_1}}\| \sqrt{\nd \ph} e^{-\frac{\ph}{h}} \nd (\tilde{u}_2-w)\|_{\backom}. \\
\end{eqnarray*}
Then using the original Carleman estimate from [DKSU], \textcircled{1} is bounded above by
\[
\frac{C}{\sqrt{\e_1}} \left( \sqrt{h} \| e^{-\frac{\ph}{h}} \L_{A_1, q_1}(\tilde{u}_2 - w) \|_{\Om} + \| \sqrt{-\nd \ph} e^{-\frac{\ph}{h}} \nd (\tilde{u}_2 - w) \|_{\frontom} \right).
\]
The last term on the right side is zero, because $\nd (\tilde{u}_2 - w) = 0$ on $U$, and $\frontom \subset U$.  Therefore the bound on \textcircled{1} can be written as
\begin{eqnarray*}
& &    \frac{C\sqrt{h}}{\sqrt{\e_1}} \| e^{-\frac{\ph}{h}} \L_{W_1, q_1}(\tilde{u}_2 - w) \|_{\Om}\\
&=&    \frac{C\sqrt{h}}{\sqrt{\e_1}} \| e^{-\frac{\ph}{h}} ((W_1-W_2) \cdot D\tilde{u}_2 + D\cdot (W_1-W_2)\tilde{u}_2 + (W_1^2 - W_2^2 + q_1 - q_2) \tilde{u}_2) \|_{\Om}\\
&\leq& \frac{C\sqrt{h}}{\sqrt{\e_1}} \| e^{-\frac{\ph}{h}}((W_1-W_2) \cdot Du_2 + D\cdot (W_1-W_2) u_2 + (W_1^2 - W_2^2 + q_1 - q_2) u_2) \|_{\Om}\\
& &    + \frac{C\sqrt{h}}{\sqrt{\e_1}} \| e^{-\frac{\ph}{h}}((W_1-W_2) \cdot Du_r + D\cdot (W_1-W_2) u_r + (W_1^2 - W_2^2 + q_1 - q_2) u_r) \|_{\Om}\\
&\leq& \frac{C\sqrt{h}}{\sqrt{\e_1}} \left( \| e^{-\frac{\ph}{h}}Du_2\|_{\Om}  + \| e^{-\frac{\ph}{h}}u_2\|_{\Om}
       + \| e^{-\frac{\ph}{h}}Du_r\|_{\Om} + \|e^{-\frac{\ph}{h}}u_r \|_{\Om} \right) \\
&\leq& \frac{C\sqrt{h}}{\sqrt{\e_1}} \left( \| e^{-\frac{\ph}{h}}Du_2\|_{\Om}  + \| a_2+r_2\|_{\Om}
       + \| e^{-\frac{\ph}{h}}Du_r\|_{\Om} + \| e^{-\frac{2 \b y}{h}}b \|_{\Om} \right)\\
&\leq& \frac{C\sqrt{h}}{\sqrt{\e_1}} \left( \| e^{-\frac{\ph}{h}}Du_2\|_{\Om}  + O(1)
       + \| e^{-\frac{\ph}{h}}Du_r\|_{\Om} + O(h) \right)\\
\end{eqnarray*}
where the constant $C$ mutates as necessary to preserve the bound.  Therefore in order to bound the terms \textcircled{1},\textcircled{2}, and \textcircled{4}, we need to calculate 
\[
\| e^{\frac{\ph}{h}} D u_1 \|_{\Om}, \| e^{-\frac{\ph}{h}}Du_2\|_{\Om}, \mbox{ and } \| e^{-\frac{\ph}{h}}Du_r\|_{\Om}.
\]
We have
\begin{eqnarray*}
\| e^{\frac{\ph}{h}} D u_1 \|_{\Om} &=&        \|e^{\frac{\ph}{h}}\frac{1}{h}D(-\ph+i\psi)e^{\frac{1}{h}(-\ph+i\psi)}(a_1+r_1)+e^{\frac{i\psi}{h}}D(a_1+r_1) \|_{\Om}\\
                                    &\lesssim& \frac{1}{h}\|D(-\ph+i\psi)(a_1 + r_1)\|_{\Om} + \| D(a_1 + r_1) \|_{\Om}\\
                                    &=&        O(h^{-1}) + O(1) \\
                                    &=&        O(h^{-1})
\end{eqnarray*}
since $\|r_1\|_{H^1(\Om)}$ is $O(h)$. Similarly,
\[
\| e^{-\frac{\ph}{h}} D u_2 \|_{\Om} = O(h^{-1}).
\]
Finally,
\begin{eqnarray*}
\| e^{-\frac{\ph}{h}}Du_r\|_{\Om} &=&        \| e^{-\frac{\ph}{h}} \frac{1}{h}D \ell e^{\frac{\ell}{h}}b + e^{-\frac{\ph}{h}} e^{\frac{\ell}{h}}Db\|_{\Om} \\
                                  &\lesssim& \frac{1}{h} \| e^{-\frac{k}{h}} b D \ell\|_{\Om} + \| e^{-\frac{k}{h}} Db\|_{\Om} \\
                                  &=&        O(1) + O(h) \\
                                  &=&        O(1). \\
\end{eqnarray*} 

Putting all of this together gives
\begin{eqnarray*}
\mbox{\textcircled{1}} &=& O(h^{-\half})  \\
\mbox{\textcircled{2}} &=& O(h^{-1})  \\
\mbox{\textcircled{3}} &=& O(1)  \\
\mbox{\textcircled{4}} &=& O(1)  \\
\mbox{\textcircled{5}} &=& O(h)  
\end{eqnarray*}

Therefore multiplying \eqref{bGf} through by $h$ and taking the limit as $h$ goes to zero gives
\[
\lim_{h \rightarrow 0} h \int_{\Om}(W_2 - W_1) \cdot(D u_2 \overline{u_1} + u_2 \overline{Du_1})dV = 0.
\]
From here on the proof follows verbatim from [DKSU]; it is outlined here only for convenience.  

We know from the proofs of Propositions \ref{boundarysolns} and \ref{noboundarysolns} that $a_1$ and $a_2$ have the form $a_1 = \g e^{\Phi_1}$ and $a_2 = e^{\Phi_2}$, where $\g$ solves
\[
(-\grad \ph + i \grad \psi) \cdot \grad \g = 0
\]
on $\Om$.  Therefore
\begin{equation}\label{Re1}
\int_{\Om} (W_2 - W_1) \cdot (\grad \ph + i \grad \psi) e^{\overline{\Phi_1} + \Phi_2}g dV = 0.  
\end{equation}
where $g$ solves
\[
(\grad \ph + i \grad \psi) \cdot \grad g = 0.
\]
on $\Om$.

Meanwhile, $\Phi_1$ and $\Phi_2$ solve
\[
(\grad \ph - i \grad \psi) \cdot \grad \Phi_1 + i(\grad \ph - i \grad \psi) \cdot W_1 + \half \Lap(\ph - i\psi)= 0
\]
and
\[
(\grad \ph + i \grad \psi) \cdot \grad \Phi_2 + i(\grad \ph + i \grad \psi) \cdot W_2 + \half \Lap(\ph + i\psi) = 0
\]
on $\Om$, respectively.  Conjugating the first equation and adding it to the second gives
\begin{equation}\label{Phi12}
(\grad \ph + i \grad \psi) \cdot (\grad(\overline{\Phi_1} + \Phi_2) + i(W_2 - W_1)) + \Lap(\ph + i\psi) = 0.  
\end{equation}
Therefore
\[
\grad \cdot ((\grad \ph + i \grad \psi)e^{\overline{\Phi_1} + \Phi_2}) = -i(W_2 - W_1) \cdot (\grad \ph + i \grad \psi) e^{\overline{\Phi_1} + \Phi_2}.
\]
Combining this with \eqref{Re1} gives
\begin{equation}\label{Re2}
\int_{\Om} g \grad \cdot ((\grad \ph + i \grad \psi)e^{\overline{\Phi_1} + \Phi_2}) dV = 0
\end{equation}
whenever $(\grad \ph + i\grad \psi) \cdot \grad g = 0$.

Choose cylindrical coordinates $s,t,$ and $\eta$ on $\RnI$ such that $s = x_{\omega}$, $t = |x'| > 0$, and $\eta = \frac{x'}{|x'|}$, and let $z$ be the complex variable $s + it$.  Then \eqref{Re2} becomes
\[
\int_{\Om} \frac{g}{z} (\partial_{\overline{z}} - \frac{n-1}{z - \overline{z}}) e^{\overline{\Phi_1} + \Phi_2} (z - \overline{z})^{n-1} dV = 0
\]
where
\[
\partial_{\overline{z}} g = 0
\]
on $\Om$.  Since this is the only restriction on $g$, we can pick any $g$ of the form $g_1(z)g_2(\eta)$ where 
\[
\partial_{\overline{z}} g_1(z) = 0
\]
on $\Om$.  Therefore if $\Om_{\eta}$ represent the slices of constant $\eta$, then
\[
\int_{\Om_{\eta}} \frac{g}{z} (\partial_{\overline{z}} - \frac{n-1}{z - \overline{z}}) e^{\overline{\Phi_1} + \Phi_2} (z - \overline{z})^{n-1} d\overline{z} \wedge dz = 0
\]
for a.e. $\eta$, whenever $g(z)$ is holomorphic on $\Om_{\eta}$.  Actually, since this holds for any holomorphic $g$ on $\Om_{\eta}$, we can write that
\[
\int_{\Om_{\eta}} g (\partial_{\overline{z}} - \frac{n-1}{z - \overline{z}}) e^{\overline{\Phi_1} + \Phi_2} (z - \overline{z})^{n-1} d\overline{z} \wedge dz = 0
\]
for a.e. $\eta$, whenever $g(z)$ is holomorphic on $\Om_{\eta}$.

Now 
\begin{eqnarray*}
& & d(g(z) e^{\overline{\Phi_1} + \Phi_2} (z - \overline{z})^{n-1} dz) \\
&=& \left( \partial_{\overline z} - \frac{n-1}{z - \overline{z}} \right) e^{\overline{\Phi_1} + \Phi_2} g(z) (z - \overline{z})^{n-1} d\overline{z} \wedge dz\\
\end{eqnarray*}
so by Stokes' theorem, 
\begin{equation}\label{bndryRe1}
\int_{\partial \Om_{\eta}} g(z) e^{\overline{\Phi_1} + \Phi_2} (z - \overline{z})^{n-1} dz = 0.
\end{equation}

\begin{lemma}\label{holomagic}
If \eqref{bndryRe1} holds for every $g$ which is holomorphic on $\Om_{\theta}$, then there exists a nonvanishing holomorphic function $F$ in $\Omega_{\eta}$, continuous on $\overline{\Om_{\eta}}$, such that
\[
F|_{\partial \Om_{\eta}} = (z - \overline{z})^{n-1} e^{\overline{\Phi_1} + \Phi_2}|_{\partial \Om_{\eta}}.
\]
\end{lemma}

\begin{proof}

Let $f = (z - \overline{z})^{n-1} e^{\overline{\Phi_1} + \Phi_2}$.  Define
\[
F(z) = \frac{1}{2\pi i } \int_{\partial \Om_{\eta}} \frac{f(\zeta)}{\zeta - z} d\zeta 
\]
for all $z \in \C \setminus \partial \Om_{\eta}$.  Then $F$ is holomorphic on the interior and exterior of $\Om_{\eta}$, and 
\begin{equation}\label{PSP}
\lim_{z \rightarrow z_0, z \in \Om_{\eta}} F(z) - \lim_{z \rightarrow z_0, z \notin \Om_{\eta}} F(z) = f(z_0)
\end{equation}
for $z_0 \in \partial \Om_{\theta}$.

Now for $z \notin \Om_{\eta}$, $(\zeta- z)^{-1}$ is a holomorphic function of $\zeta$ on $\Om_{\eta}$.  Therefore \eqref{bndryRe1} implies that $F(z) = 0$ for $z \notin \Om_{\eta}$, and so \eqref{PSP} means that 
\[
\lim_{z \rightarrow z_0, z \in \Om_{\eta}} F(z)  = f(z_0).
\]
Thus it remains only to prove that $F$ does not vanish on $\Omega_{\eta}$.  To see this, note first that since $F$ is holomorphic on $\Om_{\eta}$, the number of zeroes of $F$ enclosed by $\partial \Om_{\eta}$ is given by $n(F(\partial \Om_{\eta}), 0)$, where $n(\gamma, z_0)$ represents the winding number of $\gamma$ around $z_0$.  Since $F = f$ on the boundary, 
$n(F(\partial \Om_{\eta}), 0) = n(f(\partial \Om_{\eta}), 0)$, and 
\begin{eqnarray*}
n(f(\partial \Om_{\eta}), 0) &=& \int_{f(\partial \Om_{\eta})} \frac{dz}{z} \\
                               &=& \int_{\partial \Om_{\eta}} \frac{df}{f} \\
                               &=& \int_{\partial \Om_{\eta}}\partial_z (\overline{\Phi_1}+\Phi_2)dz + \partial_{\overline{z}}(\overline{\Phi_1}+\Phi_2)d\overline{z} \\
                               & & + \int_{\partial \Om_{\eta}} (n-1)(z - \overline{z})^{-1} dz - (n-1)(z-\overline{z}^{-1}d\overline{z} \\
                               &=& \int_{\partial \Om_{\eta}} d(\overline{\Phi_1}+\Phi_2 + (n-1)\log(z-\overline{z})\\
                               &=& 0
\end{eqnarray*}
by Stokes' theorem.  Therefore $F$ is nonvanishing in $\Om_{\eta}$.

\end{proof}

%************************************************************************************************************************************************************

This argument works for a.e. $\eta$, so we can develop a function $F(z,\eta)$ on $\Om$.  Then $\frac{dF}{F}$ is closed on $\Om$, so it is exact, and therefore $\frac{dF}{F} = da$ for some function $a$ satisfying $F = e^{a}$.  This defines a holomorphic logarithm of $F$.

Therefore
\[
\int_{\partial \Om_{\eta}} g(z)(\log(z - \overline{z})^{n-1}+\overline{\Phi_1}+\Phi_2) dz = \int_{\partial \Om_{\eta}} g(z)\log F dz = 0.
\]
Then by Stokes' theorem,
\[
\int_{\Om_{\eta}} g(z)\left( \partial_{\overline{z}}(\overline{\Phi_1}+\Phi_2) - \frac{n-1}{z - \overline{z}} \right) d\overline{z} \wedge dz = 0.  
\]
Using \eqref{Phi12}, we can rewrite this as
\begin{equation}\label{cRe}
\int_{\Om_{\eta}} g(z)(W_1-W_2) \cdot (e_s + i e_t) d\overline{z} \wedge dz = 0.  
\end{equation}
Setting $g = 1$ and separating this into real and imaginary parts gives that
\[
\int_{\Om_{\eta}} (W_1-W_2) \cdot e_s ds dt = 0
\]
and
\[
\int_{\Om_{\eta}} (W_1-W_2) \cdot e_t ds dt = 0
\]
separately.  This holds for a.e. $\eta$, where $\Om_{\eta}$ are the intersections of $\Om$ with translations of $P = \mathrm{span} \{e_s, e_t \}$.  Therefore
\begin{equation}\label{zeroslice}
\int_{v_0 + P} 1_{\Om} (W_1 - W_2) \cdot \xi dm = 0
\end{equation}
for all $\xi \in P$, where $dm$ represents the Lebesgue measure on $v_0 + P$. 

Now we can vary our choice of coordinates to move the origin around in a small neighbourhood, and vary $\omega$ slightly, without changing the fact that $E$ coincides with a graph of the form $r = f(\th)$ for smooth $f$.  Therefore \eqref{zeroslice} holds for each plane $P$ in a small neighbourhood of the original.  Then Lemma 5.2 from [DKSU] gives us that $dW_1 = dW_2$.

It remains only to prove that $q_1 = q_2$.  Note that $dW_1 = dW_2$ implies that $W_1 = W_2 + \grad \Psi$ for some function $\Psi$.  Therefore by \eqref{cRe} 
\[
\int_{\Om_{\eta}} g(z)  \partial_{\overline{z}} \Psi(z, \eta) d\overline{z} \wedge dz = 0
\]
for a.e. $\eta$, whenever $g$ is holomorphic in $\Om_{\eta}$.

By reasoning as in Lemma \ref{holomagic}, there exists a holomorphic $\tilde{\Psi}$ on $\Om_{\eta}$ such that $\tilde{\Psi} = \Psi$ on $\partial \Om_{\eta}$.  $\Psi$ is real, since $W_1$ and $W_2$ are, and so $\tilde{\Psi}$ is real on $\partial \Om_{\eta}$.  Therefore the imaginary part of  $\tilde{\Psi}$ is harmonic and zero on $\partial \Om_{\eta}$, so it must be identically $0$. Then the real part of $\tilde{\Psi}$ (and thus $\tilde{\Psi}$ itself) is constant.  Therefore $\Psi$ is constant on each  $\partial \Om_{\eta}$.

Varying $\omega$ and the origin as before, we get that $\Psi$ is constant on $\partial \Om$, and since it is only defined up to a constant anyway, we can assume that $\Psi = 0$ on $\partial \Om$.  Therefore, up to a gauge transformation, $\Psi = 0$, so $W_1 = W_2$.  Going back to \eqref{bGf}, we now have
\begin{equation}\label{newbGf}
\int_{\partial \Om \setminus U}\nd (\tilde{u}_2-w)\overline{u_1}dS = \int_{\Om} (q_2 - q_1) u_2 \overline{u_1} dx + \int_{\Om} (q_2 - q_1) u_r \overline{u_1} dV.
\end{equation}

The first and second terms on the right side are $O(1)$ and $O(h)$ as before.  The left side is now bounded by
\[
\frac{\sqrt{h}}{\sqrt{\e_1}} \left( \|e^{-\frac{\ph}{h}}(q_1 - q_2) u_2\|_{\Om} + \|e^{-\frac{\ph}{h}}(q_1 - q_2) u_r\|_{\Om} \right) = \sqrt{h}(O(1)+O(h)) = O(h^{\half}),
\]
so taking the limit of \eqref{newbGf} as $h$ goes to zero gives
\[
\lim_{h \rightarrow 0} \int_{\Om} (q_2 - q_1) u_2 \overline{u_1} dV = 0.
\]
Using the explicit forms of $u_1$ and $u_2$ gives 
\[
\int_{\Om} (q_2 - q_1)e^{\overline{\Phi_1}+\Phi_2} g dV = 0
\]
whenever $g$ solves
\[
(\grad \ph + i \grad \psi) \cdot \grad g = 0
\]
on $\Om$.  

Choose coordinates $s,t,$ and $\eta$ as before.  Again we can pick any $g$ of the form $g_1(s,t)g_2(\theta)$ where 
\[
(\partial_s + i \partial_t) g_1(s,t) = 0.
\]
on $\Om$.  Therefore if $\Om_{\eta}$ represent the slices of constant $\eta$, then
\[
\int \int_{\Om_{\eta}} g_1(s,t) (q_2 - q_1)e^{\overline{\Phi_1}+\Phi_2} ds dt = 0
\]
for a.e. $\theta$.  If $z = s + it$, then we can rewrite this as
\[
\int \int_{\Om_{\theta}} g(z) (q_2 - q_1)e^{\overline{\Phi_1}+\Phi_2} d\overline{z} \wedge dz = 0 
\]
for a.e. $\theta$, whenever $g(z)$ is holomorphic on $\Om_{\theta}$.  Now the equations for $\Phi_1$ and $\Phi_2$ read 
\[
\partial_{z} \Phi_1 = 0
\] 
and
\[
\partial_{\overline{z}} \Phi_2 = 0,
\]
so
\[
\partial_{\overline{z}}(\overline{\Phi}_1 + \Phi_2) = 0,
\]
or in other words, $\overline{\Phi}_1 + \Phi_2$ is holomorphic.  Therefore $e^{-(\overline{\Phi}_1 + \Phi_2)}$ is holomorphic, and so in particular
\[
\int \int_{\Om_{\theta}}(q_2 - q_1) d\overline{z} \wedge dz = 0 
\]

By varying $\ph$ as before, we get 
\[
\int_{v_0 + P} 1_{\Om} (q_2 - q_1)dm = 0
\]
for all planes $P$ through the origin containing a vector $v$ in a neighbourhood of $e_y$.  Then reasoning as in [DKSU] gives $q_1 = q_2$ on $\Om$.

\section{References}

\begin{enumerate}

\item[$\lbrack$DKSU$\rbrack$]
Dos Santos Ferreira, D., Kenig, C.E., Sj\"{o}strand, J., and Uhlmann, G.  Determining a Magnetic Schr\"{o}dinger Operator from Partial Cauchy Data.  \textit{Comm. Math. Phys.} 271, 467-488 (2007).

\item[$\lbrack$KSU$\rbrack$]
Kenig, C.E., Sj\"{o}strand, J., and Uhlmann, G.  The Calder\'{o}n Problem with Partial Data.  \textit{Ann. of Math.} 165, 561-591 (2007).

\item[$\lbrack$NSuU$\rbrack$]
Nakamura, G., Sun, Z., and Uhlmann, G.  Global identifiability for an inverse problem for the Schr\"{o}dinger equation in a magnetic field.  \textit{Math. Ann.} 303, 377-388 (1995).

\item[$\lbrack$Sa1$\rbrack$]
Salo, M. Inverse problems for nonsmooth first order perturbations of the Laplacian.  \textit{Ann. Acad. Scient. Fenn. Math. Dissertations} 139 (2004).

\item[$\lbrack$Sa2$\rbrack$]
Salo, M. Semiclassical pseudodifferential calculus and the reconstruction of a magnetic field.  \textit{Comm. PDE} 31, 1639-1666 (2006).

\item[$\lbrack$Su$\rbrack$]
Sun, Z.  An inverse boundary value problem for the Schr\"{o}dinger operator with vector potentials.  \textit{Trans. Amer. Math. Soc.} 338(2), 953-969 (1992).

\item[$\lbrack$To$\rbrack$]
Tolmasky, C.F.:  Exponentially growind solutions for nonsmooth first-order perturbations of the Laplacian.  \textit{SIAM J. Math. Anal.} 29(1), 116-133 (1998).

\end{enumerate}

\end{document}